\documentclass{elsarticle}
\usepackage{amsmath,amssymb,amsfonts,caption}
\usepackage{graphicx}
\usepackage{bbm}
\usepackage{pstricks}
\usepackage{euscript}
\usepackage{enumerate}
\usepackage{verbatim}
\usepackage{epsfig}
\usepackage{booktabs}
\usepackage{subcaption}
\usepackage{hyperref,cleveref}

\usepackage{fancyhdr}

\makeatletter
\def\ps@pprintTitle{%
   \let\@oddhead\@empty
   \let\@evenhead\@empty
   \def\@oddfoot{\reset@font\hfil\thepage\hfil}
   \let\@evenfoot\@oddfoot
}
\makeatother

\fancypagestyle{plain}{%
  \fancyhead[L]{\scriptsize This is a preprint of a paper whose final and definite form is with \emph{Appl. Math. Comput.}, ISSN 0096-3003, available at [\url{https://www.journals.elsevier.com/applied-mathematics-and-computation}] Paper submitted 23-November-2017; Revised 9-February-2018; Accepted 14-February-2018.}%
}

\newtheorem{theorem}{Theorem}

\newproof{proof}{Proof}

\newcommand{\lbd}{\lambda}
\renewcommand{\phi}{\varphi}

\newcommand{\dt}{\, dt}

\newcommand{\ham}{\mathcal{H}}

\newcommand{\beqn}{\begin{eqnarray}}
\newcommand{\eeqn}{\end{eqnarray}}
\newcommand{\beqnn}{\begin{eqnarray*}}
\newcommand{\eeqnn}{\end{eqnarray*}}
\newcommand{\parc}[1]{\ensuremath{\left(#1\right)}}

\DeclareMathOperator{\e}{e}

\renewcommand{\le}{\leqslant}
\renewcommand{\ge}{\geqslant}
\usepackage{dsfont}
   
   \newcommand{\R}{\ensuremath{\mathds R}}

\date{\today} 
\begin{document}
\begin{frontmatter}
\title{Optimal control of the customer dynamics based on marketing policy}
\author[mat1]{S. Rosa\corref{cor1}}
\address[mat1]{Departamento de Matem\'atica and Instituto de Telecomunica\c{c}\~{o}es, Universidade da Beira Interior, 6201-001 Covilh\~a,
   Portugal}
\ead{rosa@ubi.pt}
\author[mat2]{P. Rebelo}
\ead{rebelo@ubi.pt}
\author[mat2]{C. M. Silva}
\ead{csilva@ubi.pt}
\address[mat2]{Departamento de Matem\'atica and CMA-UBI, Universidade da Beira Interior, 6201-001 Covilh\~a, Portugal}
\author[gest1]{H. Alves}
\address[gest1]{Departamento de Gest\~ao e Economia and NECE-UBI, Universidade da Beira Interior, 6201-001 Covilh\~a, Portugal}
\ead{halves@ubi.pt}
\author[gest2]{P. G. Carvalho}
\address[gest2]{CIDESD-UBI, Universidade da Beira Interior, 6201-001 Covilh\~a, Portugal}
\ead{pguedes@ubi.pt}

\cortext[cor1]{Corresponding author.}


\begin{abstract}
We consider an optimal control problem for a non-autonomous model of ODEs that describes the evolution of the number of customers in some firm. Namely we study the best marketing strategy. Considering a $L^2$ cost functional, we establish the existence and uniqueness of optimal solutions, using an inductive argument to obtain uniqueness on the whole interval from local uniqueness. We also present some simulation results, based on our model, and compare them with results we obtain for an $L^1$ cost functional. For the $L^1$ cost functional the optimal solutions are of bang-bang type and thus easier to implement, because at every moment possible actions are chosen from a finite set of possibilities. For the autonomous case of $L^2$ problem, we show the effectiveness of the optimal control strategy against other formulations of the problem with simpler controls.
\end{abstract}
\end{frontmatter}\thispagestyle{plain}
\section{Introduction}

Firms spend millions of euros on marketing budgets. The CMO report  conducted in 2017 by the Fuqua School of Business, the American Marketing Association and Delloite shows that firms allocate, in general, between 10\% and 20\% of their re\-ve\-nu\-es on marketing budgets, depending on the sector where they operate. Con\-si\-de\-ring the high amounts involved, it is very important to optimize that allocation. However, as stated by Gupta and Steenburgh \cite{gupta2008} allocating marketing resources is a complex decision that until recently has been done based on very simple heuristics or decision rules.

Among marketing decisions and strategy is the decision to invest in referrals programs. These programs encourages current customers to recruit new customers based on rewards \cite{schmitt2011}. Contrary to other marketing programs purely based on spontaneous word-of-mouth, referral programs are marketer directed with possibility to control message content \cite{berman2016}. However, studies that help marketers to decide about the resource allocation to referral programs are scarce.

For decades, firms have been searching for the best way to maximize profits and reduce costs. Classical models usually look for ways that help firms allocating their marketing resources while maximizing profits 
\cite{albadvi2011robust}.
However, more recently models have tried to maximize customer equity (the net present value of the future profit flow over a customer’s lifetime \cite{payne2001diagnosing} 
through an optimal marketing resources allocation \cite{kumar2006clv}). 
In this sense, and based on the assumption that the number of customers in a market is limited, it is important to attract/capture new customers at earliest as possible as, otherwise, they can be attracted/captured by competitors. At the same time, a customer late attraction/caption will also reduce their customer equity.

 Following the growing interest of social networks by product marketing ma\-na\-gers, recently the classic epidemiologic models have been applied with success to specific marketing communication strategy, commonly referred as viral marketing. An application of epidemiology to a real-world problem  can be found in ~\cite{rodrigues2016}.

Previously, in ~\cite{Silva2016} the authors of this work  proposed a compartmental model suitable to describe the dynamics of the number of customers of a given firm. That model was given by a system of ordinary differential equations whose variables correspond to groups of customers and potential customers divided according to their profile and whose parameters reflect the structure of the underlying social network and the marketing policy of the firm.  Understand the flows between these groups and its consequences on the raise of customers of the firm was the main goal. Highlight the usefulness of these models in helping firms deciding their marketing policy was another objective.

Election campaign managers and companies marketing products/services managers, are interested in spreading a message by a given deadline, using limited resources. So, the optimal resource allocation over the time of the campaign is required and the formulation of such situation as an optimal control problem is suggested. In ~\cite{Kandhway2014}, that problem is tackled using two epidemic models, a SIS and a SIR.

In this paper, we consider a modified version of model ~\cite{Silva2016}, governed by the system of ordinary differential equations:
\begin{equation}\label{eq:modelo0}
\begin{cases}
\dot R= -\lbd_2 R+\lbd_1 C -\gamma(t) R +\alpha\, \beta(t) P R/N \\
\dot C= -\lbd_1 C+\lbd_2 R -\gamma(t) C + (1-\alpha)\beta(t) PR/N \\
\dot P= - \beta(t) PR/N +\gamma(t) (R + C)
\end{cases}
\end{equation}
with initial conditions
\begin{equation}\notag 
R(0), C(0), P(0)\geqslant 0,
\end{equation}
where  $R$ is the number of referral customers, $C$ is  the number of regular customer, $P$ is the number of potential customers  and $N=R+C+P$.

The parameters of the model represent the following: $\lambda_1$ is the natural transition rate between regular customers and referral costumers, given by the number of regular customers that become referral customers without external influence over the number of regular customers (by ``without external influence" we mean without being influenced by referral customers); $\lambda_2$ is the natural transition rate between referral costumers and regular customers, given by the number of referral customers that become regular customers without external influence over the number of referral customers; $\gamma(t)$ is the time varying customer defection rate, equal to the number of customers that cease to be customers over the number of customers (we assume that this rate is the same among regular and referral costumers); $\beta(t)$ is the pull effect due to marketing campaigns, corresponding to the quotient of the outcome of marketing campaigns by the number of potential customers (by ``outcome of marketing campaigns'' it is meant the number of potential customers that become customers in the sequence of marketing campaigns per unitary marketing cost per time unit); finally, $\alpha$ is the percentage of referral costumers among the new customers.

The main difference between the above model and the model presented in \cite{Silva2016} is that, instead of using a single compartment corresponding to potential clients and assuming that a fixed percentage of those potential clients are referral clients, in~\cite{Silva2016}, the potential clients are divided in two subpopulations (corresponding to potential regular clients and potential referral clients).

We stress that by using time varying parameters, $\beta(t)$ and $\gamma(t)$, in (\ref{eq:modelo0}) we obtain a non-autonomous model that is potentially more realistic.  The objective of this paper is to consider an optimal control problem for such non-autonomous model.

\section{Optimal control problem}

Inspired in \cite{Kandhway2014}, we assume that the campaigner can allocate its resources in two ways. At time $t$, he can directly recruit individuals from the population with rate $u_1(t)$, to be clients (via publicity in mass media). In addition, he can incentivize clients to make further recruitments (e.g. monetary benefits, discounts or coupons to current customers who refer their friends to buy services/products from the company). This effectively increases the ``spreading rate'' at time $t$ from $\beta(t)$ to $\beta(t)+u_2(t)$ where $u_2(t)$ denotes the ``word-of-mouth" control signal which the campaigner can adjust at time $t$.

\begin{figure}[h]
\begin{picture}(345,190)(-30,-20)
\setlength{\unitlength}{.3mm}
\put(5,0){\framebox(60,45){$C$}}
\put(5,100){\framebox(60,45){$R$}}
\put(270,0){\framebox(60,145){$P$}}
\put(80,140){\vector(1,0){180}}
\put(160,145){$\gamma(t) R$}
\put(260,120){\vector(-1,0){180}}
\put(155,127){$\alpha_1\, u_1 P$}
\put(260,108){\vector(-1,0){180}}
\put(135,92){$ \alpha_2\,(\beta(t)+ u_2)  PR/N $}
\put(80,40){\vector(1,0){180}}
\put(160,45){$\gamma(t) C$}
\put(260,20){\vector(-1,0){180}}
\put(145,27){$(1-\alpha_1)u_1 P$}
\put(260,10){\vector(-1,0){180}}
\put(120,-7){$(1-\alpha_2)(\beta(t)+ u_2) PR/N$}
\put(25,50){\vector(0,1){45}}
\put(0,70){$\lambda_1 C$}
\put(35,95){\vector(0,-1){45}}
\put(40,70){$\lambda_2 R$}
\end{picture}
\caption{The compartmental model.}
\label{diagram:model}
\end{figure}
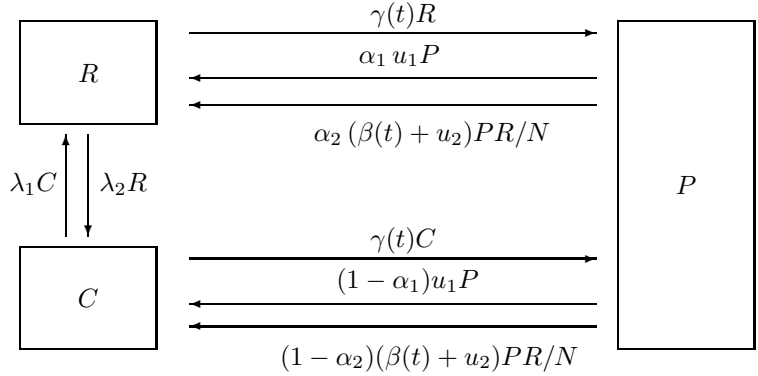

The diagram of the non-autonomous  model we propose is shown in Figure \ref{diagram:model}. The respective equations are the following:

\begin{equation}\label{eq:modelo}
\begin{cases}
\dot R= -\lbd_2 R+\lbd_1 C -\gamma(t) R+\alpha_1\, u_1 P +\alpha_2\, (\beta(t)+u_2) P R/N \\
\dot C= -\lbd_1 C+\lbd_2 R -\gamma(t) C + (1-\alpha_2)(\beta(t)+u_2) PR/N + u_1(1-\alpha_1)P\\
\dot P= - (\beta(t)+u_2) PR/N-u_1 P +\gamma(t) R +\gamma(t) C
\end{cases}
\end{equation}
with initial conditions
\begin{equation}\label{eq:cond_ini}
R(0), C(0), P(0)\geqslant 0.
\end{equation}
The parameters $u_1$, $u_2$ will be taken in the space $L^{\infty}$ functions such that \linebreak $u_1\in[0,{u_1}_{\text{max}}]$ and $u_2\in[0,{u_2}_{\text{max}}]$.

Our purpose is to minimize the number of potential customers and the cost associated to the control of the marketing campaigns. To obtain the best reduction in the number of potential customers, we minimize the evolution history, $P(t)$, $0\leqslant t\leqslant t_f$. Note that minimizing the number of potential customers correspond to maximizing the number of customers (potential and referral) and that by minimizing the evolution history of potential customers, instead of the final number, we are increasing the customer equity.

We consider the optimal control problem:
\begin{equation}
\label{control-problem}
\tag{P}
\begin{gathered}
\mathcal{J}(P,u_1,u_2) =\int_0^{t_f}\kappa_1\,P+ \kappa_2\, u_1^2+\kappa_3 \, u_2^2 ~dt \quad \longrightarrow \quad \min\\
\begin{cases}
\dot R = -\lbd_2 R+\lbd_1 C -\gamma R+\alpha_1\, u_1 P +\alpha_2\, (\beta(t)+u_2) P R/N \\ \notag
\dot C = -\lbd_1 C+\lbd_2 R -\gamma C + (1-\alpha_2)(\beta(t)+u_2) PR/N + \notag u_1(1-\alpha_1)P\\ \notag
\dot P = - (\beta(t)+u_2) PR/N-u_1 P +\gamma R +\gamma C
\end{cases}\\
(C(0),R(0),P(0))=(C_0,R_0,P_0),
\end{gathered}
\end{equation}
where $0<\kappa_1,\kappa_2,\kappa_3 <\infty$ and $R_0,C_0,P_0$ are non-negative, the state variables are absolutely continuous functions, $(C(\cdot),R(\cdot),P(\cdot)) \in AC([0,t_f];\R^4)$, and the controls are Lebesgue integrable, $(u_1(\cdot),u_2(\cdot)) \in L^1([0,t_f];[0,{u_1}_{\max}]\times[0,{u_2}_{\max}])$.

In sections~\ref{section:EOS} to~\ref{section:Uniq} we show that a solution of problem~\eqref{control-problem} exists and is unique in the whole interval $[0,t_f]$. To establish the existence of solution, we use a standard result that assures the existence of an optimal control pair $(u_1^*,u_2^*)$ and a correspon\-ding solution of the initial value problem that minimizes the cost functional over $L^1([0,t_f];[0,{u_1}_{\max}]\times[0,{u_2}_{\max}])$. The fact that the optimal controls are bounded, assures that the optimal controls are in fact in $L^{\infty}([0,t_f];[0,{u_1}_{\max}]\times[0,{u_2}_{\max}])$ (see section~\ref{section:COC}).

To obtain uniqueness, we assume, by contradiction, that there are two distinct optimal
pairs of state and co-state variables
$$((C,P,R),(p_1,p_2,p_3)) \quad \text{and} \quad ((C^*,P^*,R^*),(p_1^*,p_2^*,p_3^*)),$$
which correspond to two distinct optimal controls
$\left(u,v\right)$ and
$\left(u^*, u^*\right)$, verifying \eqref{eq:Pontryagin-CPR-Mayer-6} and
\eqref{eq:Pontryagin-CPR-Mayer-7}. The existence of some compact positively invariant region $\Gamma$,
which is independent on the controls, allows us to prove that there is a contradiction unless
the state variables, the co-state variables and the optimal controls are the same on a small time interval $[0, T]$.
The next step consists in describing an iterative procedure that allows one to extend the uniqueness of
solution to the interval $[0,(k+1)T]$, assuming we have uniqueness on the interval $[0,kT]$. This allows us to conclude that we have
the required uniqueness on the whole interval after a finite number of steps.

\section{Existence of an optimal solution}\label{section:EOS}

To prove  that there is an optimal solution of problem~\eqref{control-problem},
we will use a result that ensures the existence of the solution for optimal control problems
contained in Theorem III.4.1 and Corollary III.4.1 in~\cite{Fleming-Rishel-Springer-Verlag-1975},
Theorem~\ref{teo:existence-Fleming-Raymond-1974} below.
Problem \eqref{control-problem} is an optimal control problem in Lagrange form:
\begin{equation}
\label{general-problem}
\begin{gathered}
J(x,u)=\displaystyle \int_{t_0}^{t_1} \mathcal{L}(t,x(t),u(t)) \ dt
\longrightarrow \min ,\\
\begin{cases}
x'(t)=f\left(t,x(t),u(t)\right), \quad \text{a.e.} \ t \in [t_0,t_1],\\
x(t_0)=x_0,
\end{cases}\\
x(\cdot) \in AC\left([t_0,t_1];\R^n\right),
\quad u(\cdot) \in L^1([t_0,t_1];U\subset \R^m).
\end{gathered}
\end{equation}
In the above context, we say that a pair $(x,u)
\in AC\left([t_0,t_1];\R^n\right) \times L^1([t_0,t_1];U)$
is feasible if it satisfies the Cauchy problem in~\eqref{general-problem}.
We denote the set of all feasible pairs by $\mathcal F$.
Next, we recall

\begin{theorem}[See \cite{Fleming-Rishel-Springer-Verlag-1975}]
\label{teo:existence-Fleming-Raymond-1974}
For problem~\eqref{general-problem}, suppose that $f$ and $\mathcal{L}$
are continuous and there exist positive constants $C_1$ and $C_2$ such that,
for $t \in \R$, $x,x_1,x_2 \in \R^n$ and $u \in \R^m$, we have
\begin{enumerate}[a)]
\item \label{teo:existence-Fl-Ray-1974-a}
$\|f(t,x,u)\| \le C_1(1+\|x\|+\|u\|)$;

\item \label{teo:existence-Fl-Ray-1974-b}
$\|f(t,x_1,u)-f(t,x_2,u)\| \le C_2\|x_1-x_2\|(1+\|u\|)$;

\item \label{teo:existence-Fl-Ray-1974-c}
$\mathcal F$ is non-empty;

\item \label{teo:existence-Fl-Ray-1974-d}
$U$ is closed;

\item \label{teo:existence-Fl-Ray-1974-e}
there is a compact set $S$ such that $x(t_1) \in S$
for any state variable $x$;

\item \label{teo:existence-Fl-Ray-1974-f}
$U$ is convex, $f(t,x,u)=\alpha(t,x)+\beta(t,x)u$,
and $\mathcal{L}(t,x,\cdot)$ is convex on $U$;

\item \label{teo:existence-Fl-Ray-1974-g}
$\mathcal{L}(t,x,u)\ge c_1|u|^\beta-c_2$, for some $c_1>0$ and $\beta >1$.
\end{enumerate}
Then, there exist $(x^*,u^*)$ minimizing $J$ on $\mathcal F$.
\end{theorem}

Applying Theorem~\ref{teo:existence-Fleming-Raymond-1974} to our problem we obtain the following result:

\begin{theorem}\label{teo:exist-sol}
  There exists an optimal control pair $(u_1^*,u_2^*)$ and a corresponding solution of the initial value problem in~\eqref{control-problem}, $(R^*,C^*,P^*)$, that minimizes the cost functional $\mathcal J$ in~\eqref{control-problem} over $L^1([0,t_f];[0,{u_1}_{\max}]\times[0,{u_2}_{\max}])$.
\end{theorem}

\begin{proof}
We first note that, adding the equations in~\eqref{eq:modelo}, we conclude that the total population is constant: $N(t)=C_0+R_0+P_0:=N_0$. Thus $C(t),R(t),P(t) \le N_0$. Additionally, $P(t)R(t)/N(t)\le P(t)\le N_0$. We immediately obtain~\ref{teo:existence-Fl-Ray-1974-a}) and~\ref{teo:existence-Fl-Ray-1974-b}).

Conditions ~\ref{teo:existence-Fl-Ray-1974-c}) and~\ref{teo:existence-Fl-Ray-1974-d}) are immediate from the definition of $\mathcal F$ since \linebreak $U=[0,{u_1}_{max}]\times[0,{u_2}_{max}]$.

We conclude that all the state variables are in the compact set
$$\{(x,y,z) \in (\R_0^+)^3: 0 \le x+y+z = N_0\}$$
and condition~\ref{teo:existence-Fl-Ray-1974-e}) follows.

Since the state equations are linearly dependent on the controls and $L$ is quadratic in the controls, we obtain~\ref{teo:existence-Fl-Ray-1974-f}). Finally,
$$
L = \kappa_1 I+\kappa_2 u_1^2+\kappa_3 u_2^2 \ge \min\{\kappa_2,\kappa_3\}(u_1^2+u_2^2) =
\min\{\kappa_2,\kappa_3\}\|(u_1,u_2)\|^2
$$
and we establish~\ref{teo:existence-Fl-Ray-1974-g}) with $c_1=\min\{\kappa_2,\kappa_3\}$.

Thus the result follows from Theorem~\ref{teo:existence-Fleming-Raymond-1974}.
\end{proof}

\section{Characterization of the optimal controls}\label{section:COC}

In this section, using the Pontryagin
Maximum Principle \cite{book:Pont}, we characterize the solutions that, according to Theorem~\ref{teo:exist-sol}, the solution exist.

Note that Theorem~\ref{teo:existence-Fleming-Raymond-1974} does not require $U$ to be a bounded set and thus, in general, the $L^1$ optimal controls predicted by Theorem~\ref{teo:existence-Fleming-Raymond-1974} are not necessarily in $L^\infty$. As a consequence, in general, one can not assure that the optimal controls satisfy the Pontryagin Maximum Principle, see \cite{book:Pont}.

However, in our case, the compacity of the set $[0,{u_1}_{\max}]\times[0,[{u_2}_{\max}]$ assures that the control minimizers, ${u_1}^*$ and ${u_2}^*$,
are in $L^\infty$ as required by the Pontryagin Maximum principle. Furthermore, in our context, there are no abnormal minimizers \cite{abnormal} since in our case only initial conditions are imposed and, in particular, the state variables
are free at the terminal time. Thus, we can fix the cost multiplier associated with
the Lagrangian $\mathcal{L}$ to be minus one.

The Hamiltonian associated with problem~\eqref{control-problem} is given by:
\begin{equation}\notag 
\begin{split}
& \ham(t,(C,R,P),(p_1,p_2,p_3),(u_1,u_2))\\
&  =  \kappa_1 P +\kappa_2 u_1^2+\kappa_3 u_2^2 \\
& \quad + p_1(-\lbd_2 R+\lbd_1 C -\gamma(t) R+\alpha_1\, u_1 P +\alpha_2\, (\beta(t)+u_2) P R/N )\\
& \quad +p_2(-\lbd_1 C+\lbd_2 R -\gamma(t) C + (1-\alpha_2)(\beta(t)+u_2) PR/N + u_1(1-\alpha_1)P)\\
& \quad +p_3(- (\beta(t)+u_2) PR/N-u_1 P +\gamma(t) R +\gamma(t) C)
\end{split}
\end{equation}

In what follows, we use the operator $\partial_i$ to denote
the partial derivative with respect to the $i$th variable.

\begin{theorem}[Necessary optimality conditions]
$\;$\\ 
If $((C^*,R^*,P^*),({u_1}^*,{u_2}^*))$ is a mi\-nimizer of problem~\eqref{control-problem},
then there are multipliers $(p_1(\cdot),p_2(\cdot),p_3(\cdot))
\in AC([0,t_f];\R^3)$ such that
\begin{equation}
\label{eq:Pontryagin-CPR-Mayer-1}
\begin{cases}
\dot p_1 = \lambda_2(p_1-p_2)+\gamma(t)(p_1-p_3)\\[2mm]
\quad \quad +\left[p_3-\alpha_2 p_1-(1-\alpha_2)p_2\right](\beta(t)+u_2)P(C+P)/N^2\\[2mm]
\dot p_2 = \lambda_1(p_2-p_1)+\gamma(t)(p_2-p_3)\\[2mm]
\quad \quad  -\left[p_3-\alpha_2 p_1-(1-\alpha_2)p_2\right](\beta(t)+u_2)PR/N^2\\[2mm]
\dot p_3 = -\kappa_1+\left[p_3-\alpha_1 p_1-(1-\alpha_1)p_2 \right]u_1\\[2mm]
\quad \quad + \left[p_3-\alpha_2 p_1-(1-\alpha_2)p_2 \right](\beta(t)+u_2)R(C+R)/N^2
\end{cases}
\end{equation}
for almost all $t \in [0,t_f]$, with transversality conditions
\begin{equation}
\label{eq:Pontryagin-CPR-Mayer-5}
p_1(t_f)=p_2(t_f)=p_3(t_f)=p_4(t_f)=0.
\end{equation}
Furthermore, the optimal control pair is given by
\begin{equation}
\label{eq:Pontryagin-CPR-Mayer-6}
{u_1}^*=\min\left\{\max\left\{0,\dfrac{[p_3-\alpha_1 p_1 -(1-\alpha_1)p_2]P^* }{2 \kappa_2} \right\},{u_1}_{\max}\right\}
\end{equation}
and
\begin{equation}
\label{eq:Pontryagin-CPR-Mayer-7}
{u_2}^*=\min\left\{\max\left\{0,\dfrac{[p_3 -\alpha_2 p_1-(1-\alpha_2)p_2]P^* R^*}{2 \kappa_3N^*} \right\},{u_2}_{\max}\right\}.
\end{equation}
\end{theorem}

\begin{proof}
Properties~\eqref{eq:Pontryagin-CPR-Mayer-1} and~\eqref{eq:Pontryagin-CPR-Mayer-5} are a consequence of Pontryagin Maximum Principle.

The optimality conditions on the set
$$\{t \in [0,t_f]: 0<{u_1}^*(t)<{u_1}_{\max} \ \wedge \ 0<{u_2}^*(t)<{u_2}_{\max}\}$$ yield
$$\frac{\partial \ham}{\partial {u_1}^*}=0 \quad \Leftrightarrow \quad {u_1}^*=\dfrac{p_3-\alpha_1 p_1 -(1-\alpha_1)p_2}{2 \kappa_2}P$$
and
$$\frac{\partial \ham}{\partial {u_2}^*}=0 \quad \Leftrightarrow \quad
 {u_2}^*=\dfrac{p_3-\alpha_2 p_1-(1-\alpha_2)p_2}{2 \kappa_3}\frac{PR}{N}.$$

If $t \in \text{int} \{t \in [0,t_f]: {u_1}^*(t)={u_1}_{\max}\}$, then the maximality condition is
$$\frac{\partial \ham}{\partial {u_1}^*} \le 0 \quad \Leftrightarrow \quad \dfrac{p_3-\alpha_1 p_1 -(1-\alpha_1)p_2}{2 \kappa_2}P\ge {u_1}_{\max}$$
and if $t \in \text{int} \{t \in [0,t_f]: {u_2}^*(t)={u_2}_{\max}\}$, then the maximality condition is
$$\frac{\partial \ham}{\partial {u_2}^*} \le 0 \quad \Leftrightarrow \quad \dfrac{p_3-\alpha_2 p_1-(1-\alpha_2)p_2}{2 \kappa_3}\frac{PR}{N}\ge {u_2}_{\max}.$$

Similarly, if $t \in \text{int} \{t \in [0,t_f]: {u_1}^*(t)=0\}$, then the maximality condition is
$$\frac{\partial \ham}{\partial {u_1}^*} \ge 0 \quad \Leftrightarrow \quad \dfrac{p_3-\alpha_1 p_1 -(1-\alpha_1)p_2}{2 \kappa_2}P\le 0$$
and if $t \in \text{int} \{t \in [0,t_f]: {u_2}^*(t)=0\}$, then the maximality condition is
$$\frac{\partial \ham}{\partial {u_2}^*} \ge 0 \quad \Leftrightarrow \quad \dfrac{p_3-\alpha_2 p_1-(1-\alpha_2)p_2}{2 \kappa_3}\frac{PR}{N}\le 0.$$
\end{proof}

\section{Uniqueness of solution} \label{section:Uniq}

In this section we prove the  uniqueness of the optimal solution of~\eqref{control-problem} in the whole interval  $[0, t_f]$.
The proof of this result is inspired on~\cite{Gaff-Schaefer-MBE-2009, Mateus-Rebelo-Rosa-Silva-Torres-DCDS-S-2017}. Namely, on~\cite{Gaff-Schaefer-MBE-2009} uniqueness is established in a sufficiently small interval for some autonomous epidemiological models and in~\cite{Mateus-Rebelo-Rosa-Silva-Torres-DCDS-S-2017} the result is proved for a general non-autonomous version of one of those models and uniqueness is established on the whole interval.

\begin{theorem}\label{teo:uniqueness}
The solution of the optimal control problem~\eqref{control-problem} is unique.
\end{theorem}

Theorem~\ref{teo:uniqueness} establishes the uniqueness of the optimal solution of~\eqref{control-problem} throughout the time interval
where the optimal control problem was considered, $[0, t_f]$. The proof of this result is done in two steps. Namely, on a first moment we establish the uniqueness on a sufficiently small time interval $[0,T]$ and afterwards we show that the result can be extended to the whole time interval by using an induction argument.

In more detail, to prove uniqueness on a small time interval, we use a contradiction argument adapted from the argument used in~\cite{Gaff-Schaefer-MBE-2009}, in the autonomous context, and also considered in~\cite{Mateus-Rebelo-Rosa-Silva-Torres-DCDS-S-2017}, for a nonautonomous model. We start by assuming that there are two distinct optimal pairs of state and co-state variables corresponding to two different optimal controls. Making a change of variables we are able to prove that we have a contradiction unless the state and co-state variables are the same in some sufficiently small time interval and, using the characterization of the optimal controls given by~\eqref{eq:Pontryagin-CPR-Mayer-6} and~\eqref{eq:Pontryagin-CPR-Mayer-7} we conclude that the optimal controls coincide in that small time interval $[0, T]$.

The second step in the argument, consists in noting that there are two possibilities: $T = t_f$ or $T < t_f$. In the first case the proof is completed.
Otherwise, noting that the estimates used to obtain $T$ in the first place are only related with the maximum value of the parameters and the bounds for the state and co-state variables on some invariant region that is independent on the controls and using for initial conditions at time $T$ the values of the state trajectories at the right-end of the interval $[0, T]$, we obtain uniqueness on the interval $[T, 2T]$. Iterating the procedure, after a finite number of steps, we obtain uniqueness in whole the interval $[0, t_f]$.

 The proof of Theorem \ref{teo:uniqueness} may be found in \ref{appendix_proof}.

\section{Simulation} 
\begin{table}[!htb]
\centering
{\renewcommand{\arraystretch}{1.25} 
\begin{tabular}{cc}
\toprule
Parameter & Value \\
\midrule
$\alpha_1$ & 0.05\\
$\alpha_2$ & 0.10\\
$\lambda_1$ & 0.002 \\
$\lambda_2$  & $\lambda_1C_0/R_0$\\
\bottomrule
\end{tabular}}
\caption{Values of parameters.}
\label{tab:ref}
\end{table}

\begin{figure}[!htb]
       \centering
        \hspace*{-2cm}
        \begin{subfigure}[b]{0.48\textwidth}\centering
                \includegraphics[scale=0.39]{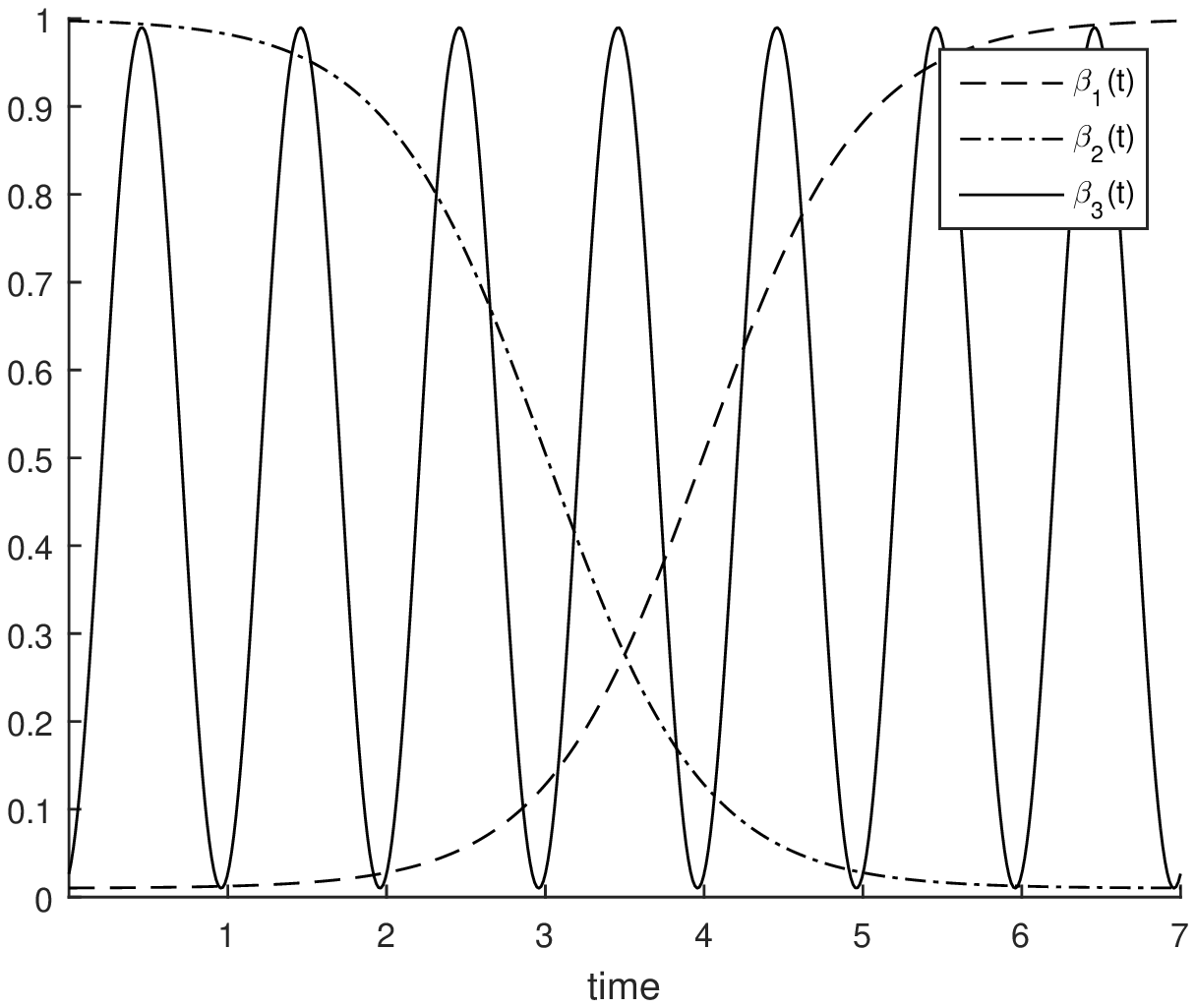}
                \caption{Functions $\beta_1(t)$, $\beta_2(t)$ and $\beta_3(t)$.}
                \label{fig:betas:3}
        \end{subfigure}
        \hspace*{1cm}
        \begin{subfigure}[b]{0.48\textwidth}\centering
                \includegraphics[scale=0.39]{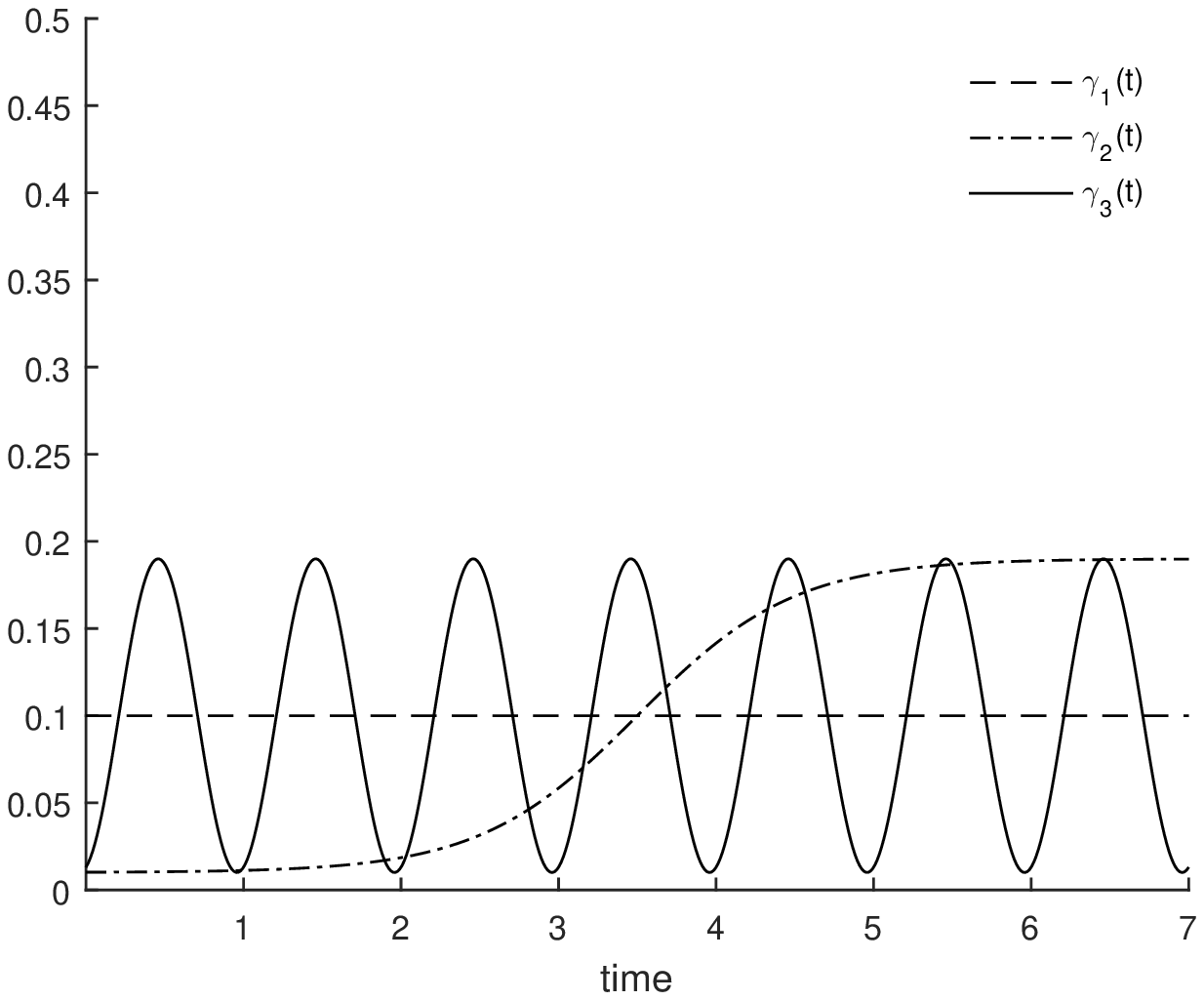}
                \caption{Functions $\gamma_1(t)$, $\gamma_2(t)$ and $\gamma_3(t)$.}
                \label{fig:gamas:3}
        \end{subfigure}\\ \centering
        \caption{Recruitment rate functions and defection rate functions. }
\end{figure}

\begin{figure}[!htb]
        \centering
        \hspace*{-2cm}
        \begin{subfigure}[b]{0.48\textwidth}\centering
                \includegraphics[scale=0.39]{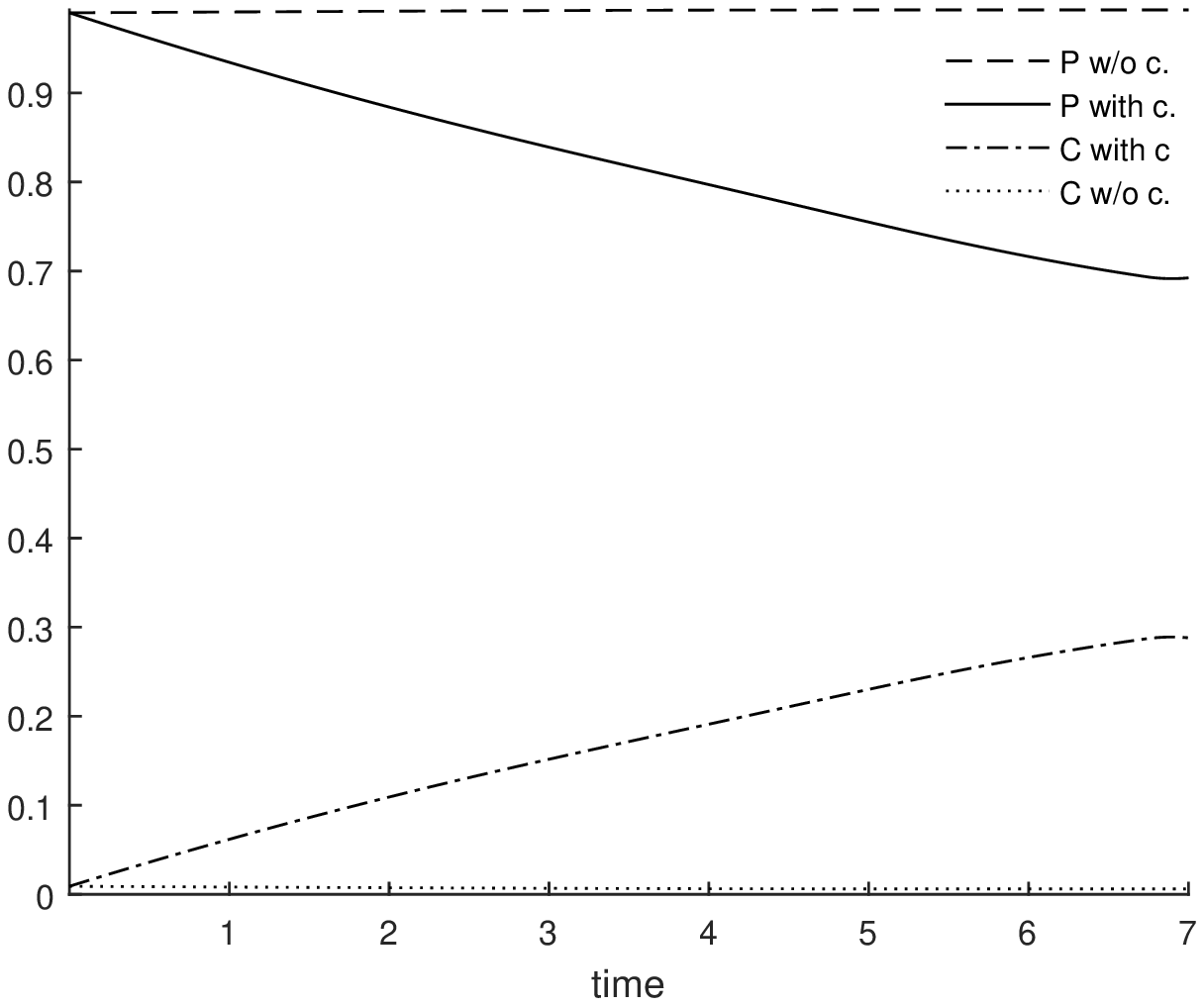}
                \caption{Evolution of $C$ and $P$. }
                \label{fig:PC_beta1}
        \end{subfigure}
        \hspace*{1cm}
        \begin{subfigure}[b]{0.48\textwidth}\centering
                \includegraphics[scale=0.39]{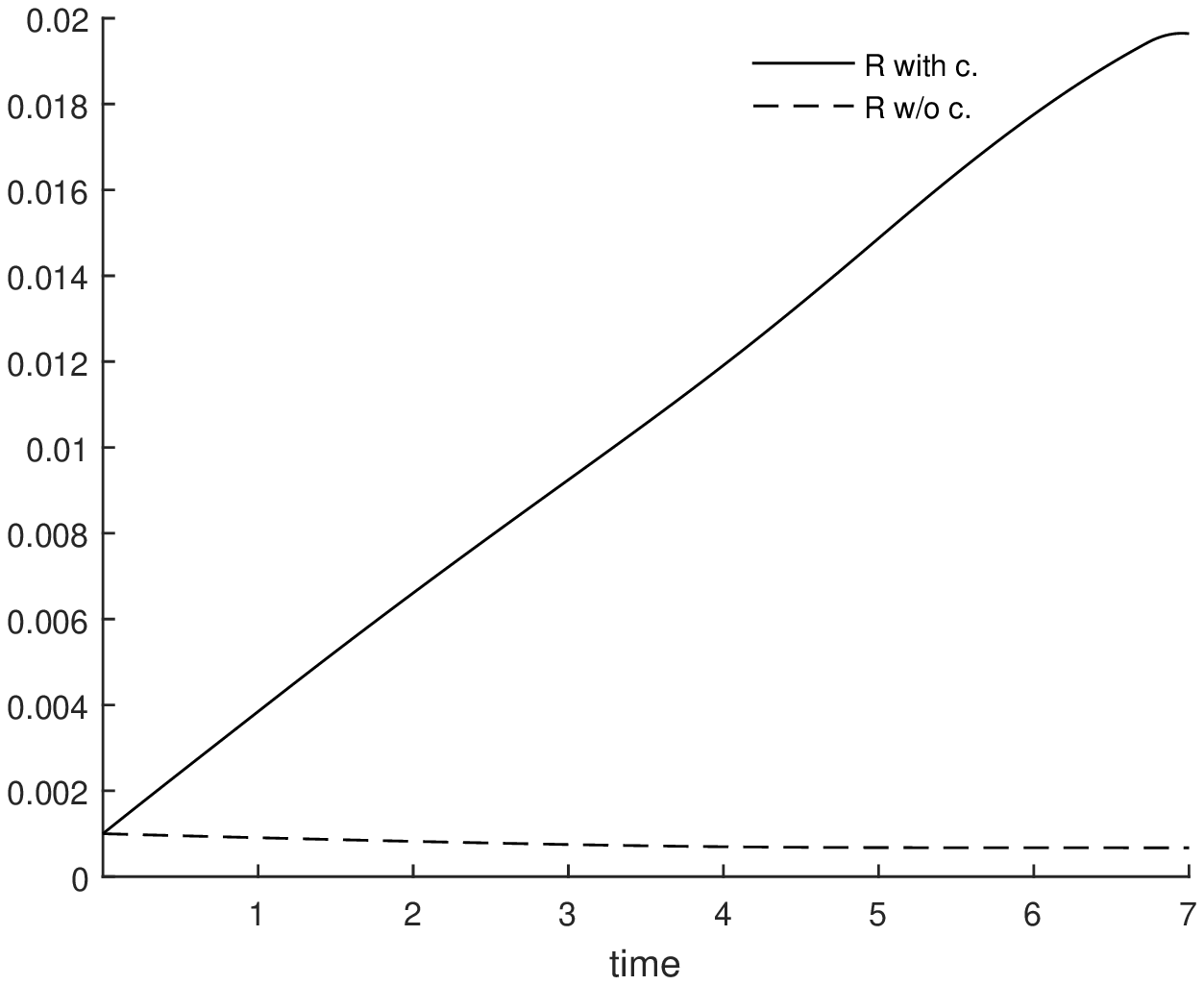}
                \caption{Evolution of $R$. }
                \label{fig:R_beta1}
        \end{subfigure}\\ \centering  \hspace*{-2cm}
        \begin{subfigure}[b]{0.48\textwidth}\centering
                \includegraphics[scale=0.39]{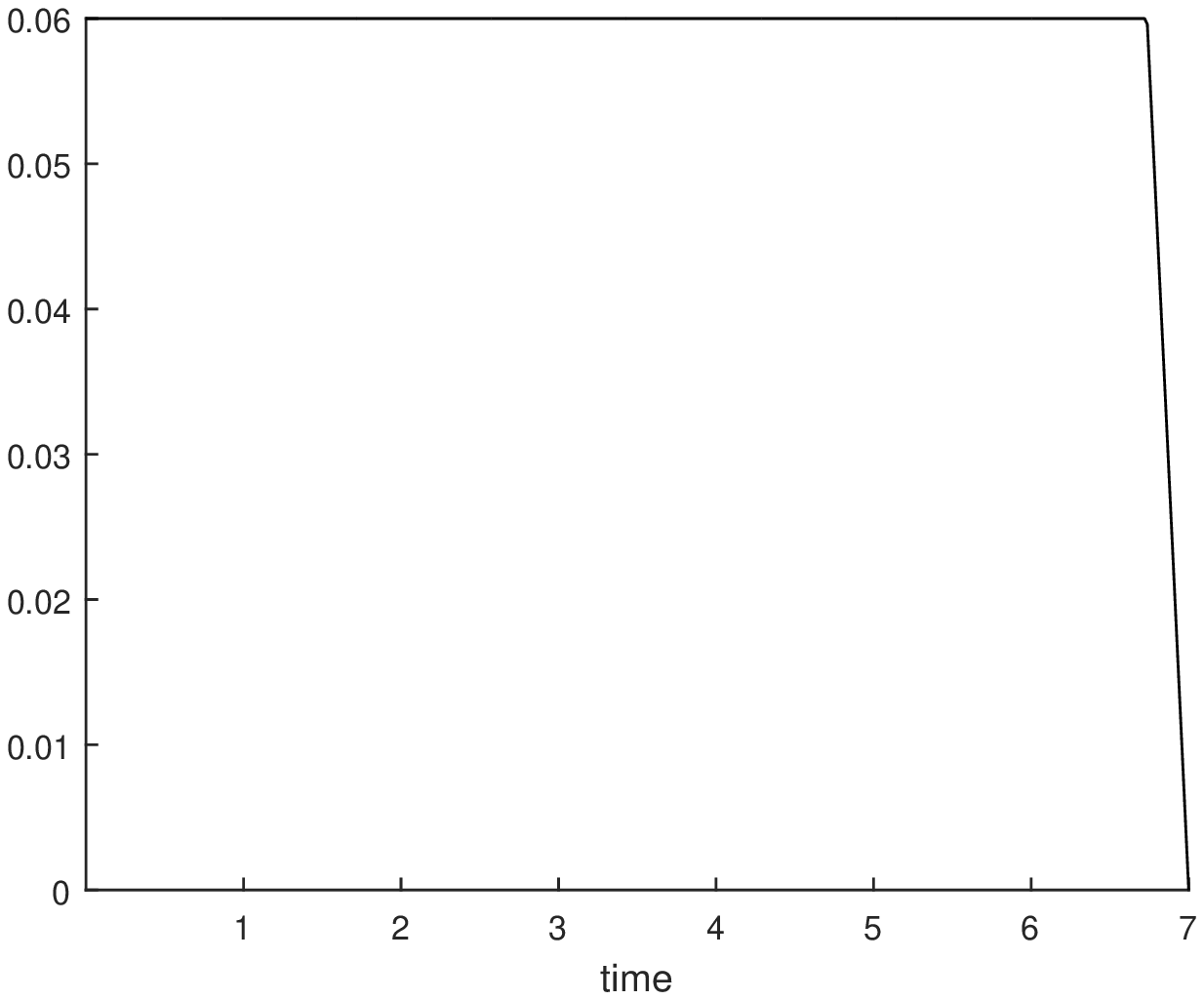}
                \caption{Variation of the optimal control $u_1$. }
                \label{fig:u1_beta1}
        \end{subfigure}
        \hspace*{1cm}
        \begin{subfigure}[b]{0.48\textwidth}\centering
                \includegraphics[scale=0.39]{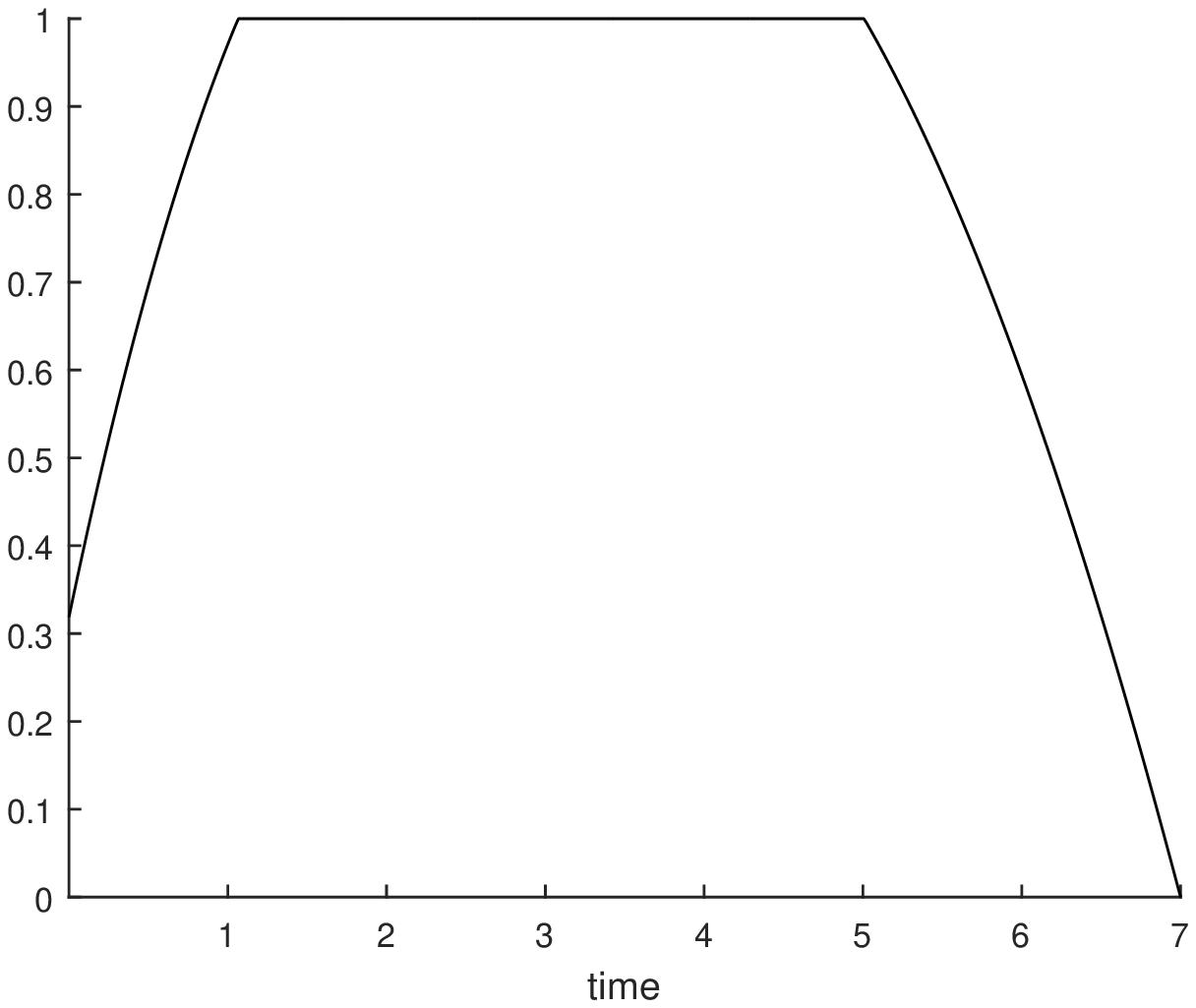}
                \caption{Variation of the optimal control $u_2$. }
                \label{fig:u2_beta1}
        \end{subfigure}\\
         \caption{Optimal control and state variables, with control and without control, of the marketing model,  time-varying rate $\beta_1$ and time-varying rate $\gamma_1$. \emph{Top row:} (A) regular customers C and potential customers P, (B) referral customers R. \emph{Bottom row:} (C) optimal control $u_1$, (D) optimal control $u_2$.}
        \label{fig:OCP_beta1}
\end{figure}

\begin{figure}[!htb]
        \centering
        \hspace*{-2cm}
        \begin{subfigure}[b]{0.48\textwidth}\centering
                \includegraphics[scale=0.39]{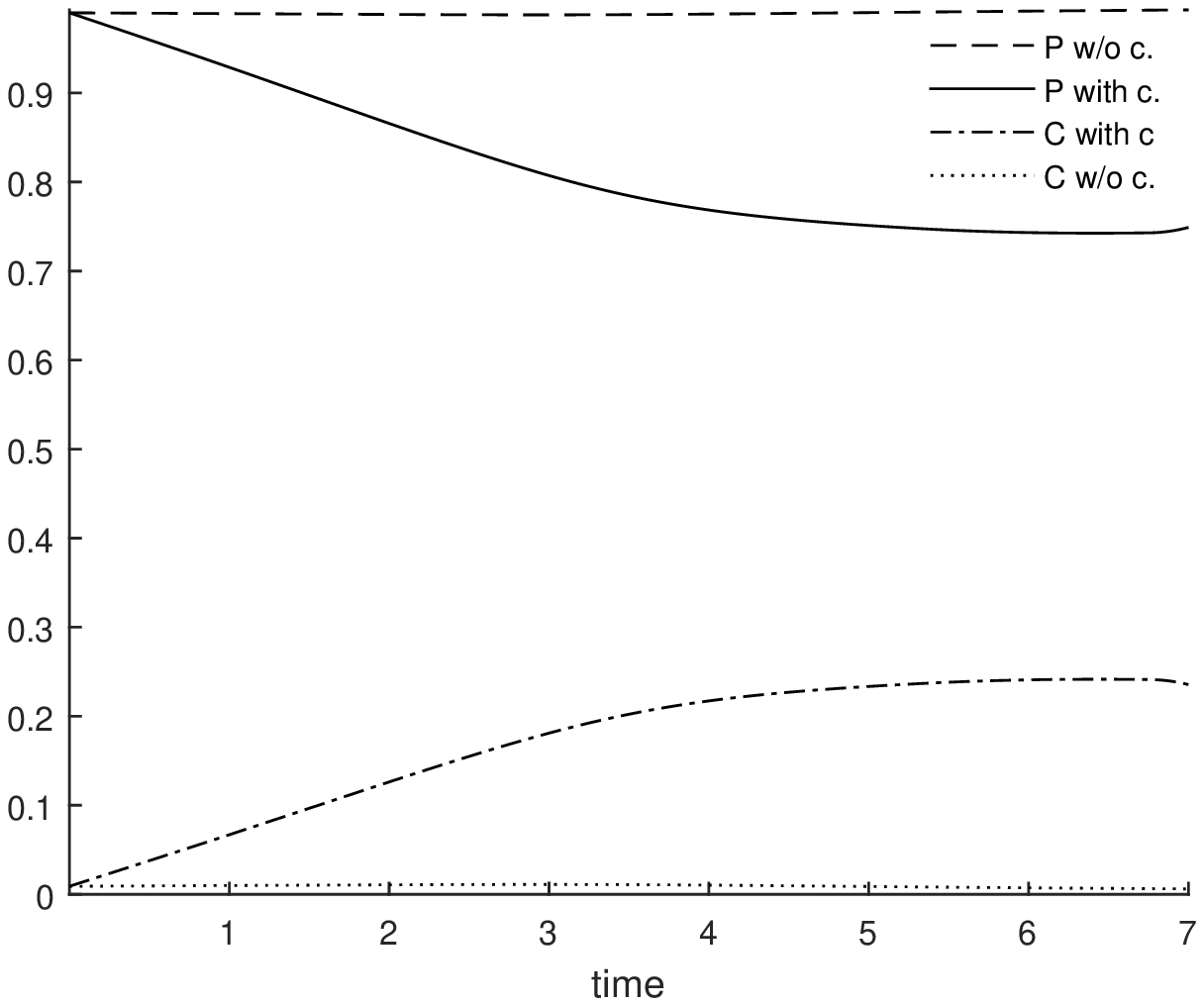}
                \caption{Evolution of $C$ and $P$. }
                \label{fig:PC_beta2}
        \end{subfigure}
        \hspace*{1cm}
        \begin{subfigure}[b]{0.48\textwidth}\centering
                \includegraphics[scale=0.39]{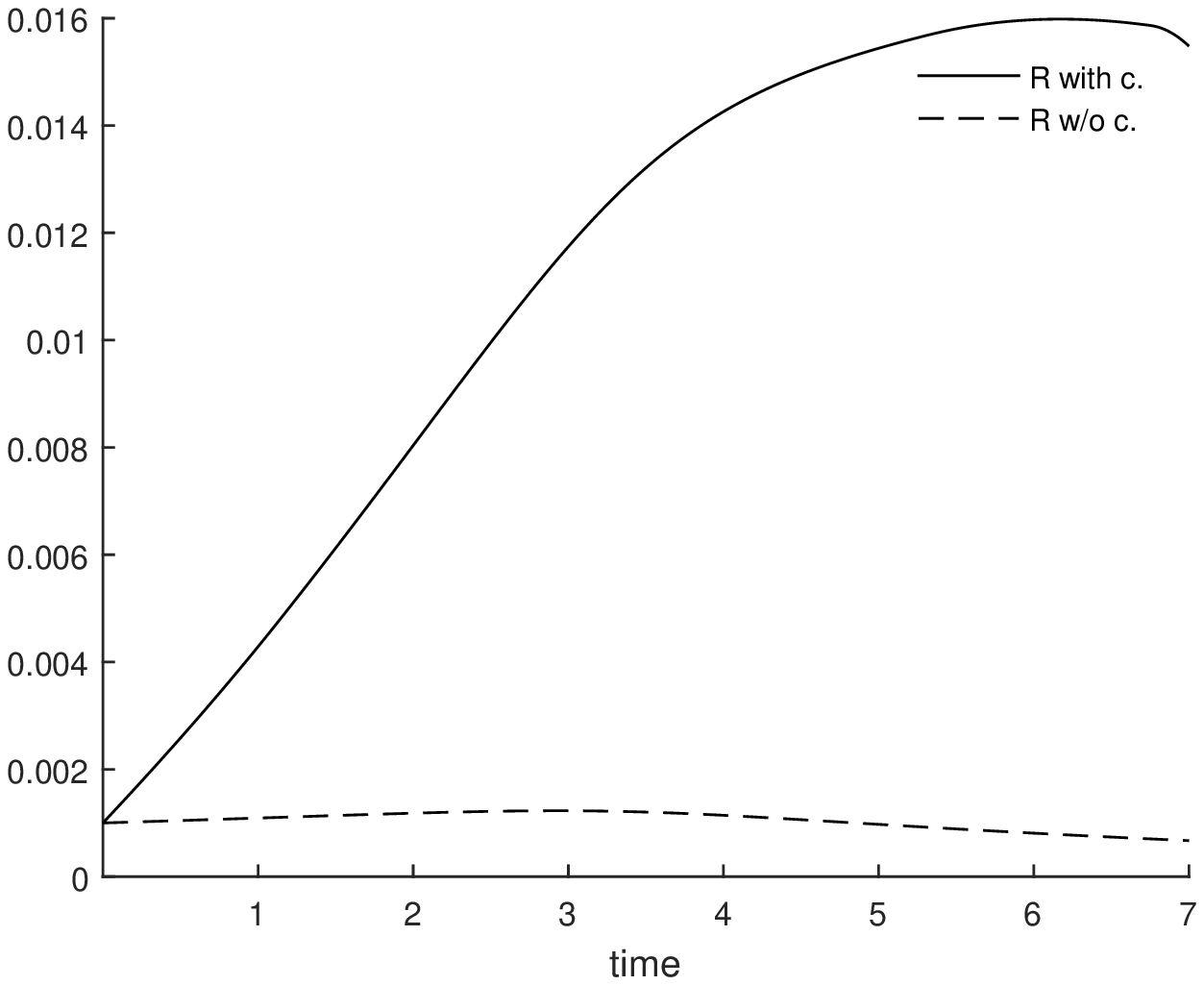}
                \caption{Evolution of $R$. }
                \label{fig:R_beta2}
        \end{subfigure}\\ \centering  \hspace*{-2cm}
        \begin{subfigure}[b]{0.48\textwidth}\centering
                \includegraphics[scale=0.39]{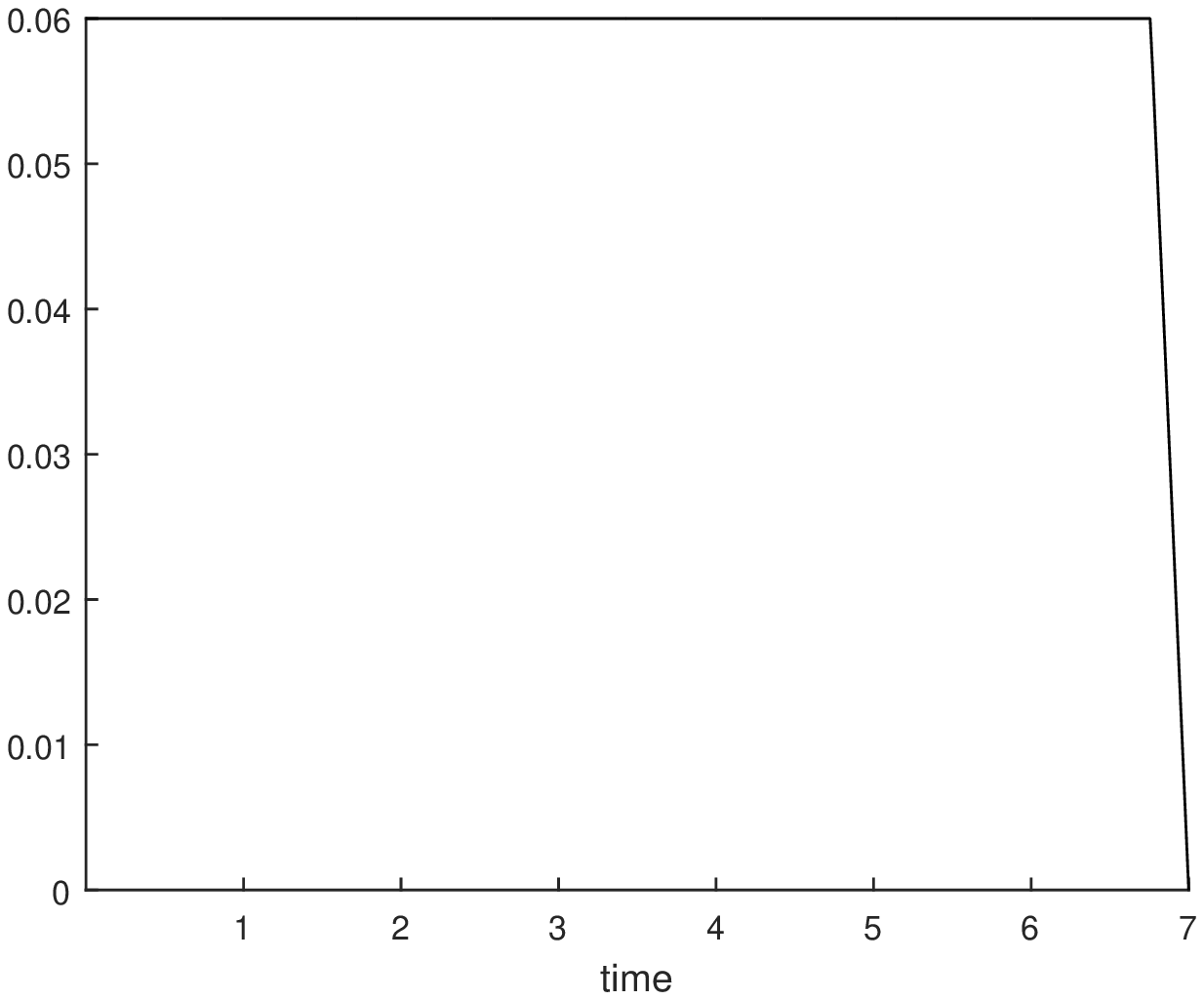}
                \caption{Variation of the optimal control $u_1$. }
                \label{fig:u1_beta2}
        \end{subfigure}
        \hspace*{1cm}
        \begin{subfigure}[b]{0.48\textwidth}\centering
                \includegraphics[scale=0.39]{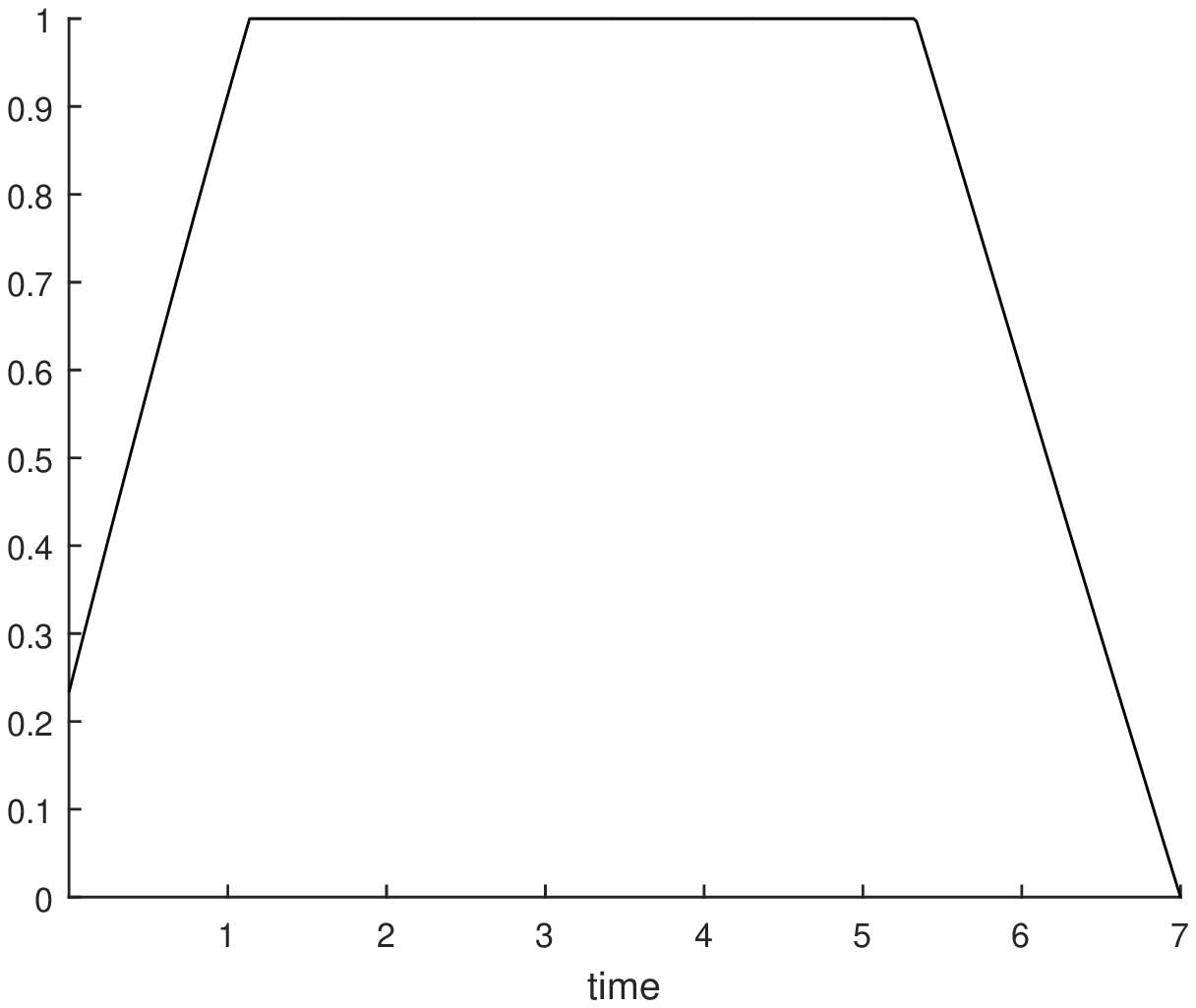}
                \caption{Variation of the optimal control $u_2$. }
                \label{fig:u2_beta2}
        \end{subfigure}\\
         \caption{Optimal control and state variables, with control and without control, of the marketing model,  time-varying rate $\beta_2$ and time-varying rate $\gamma_2$. \emph{Top row:} (A) regular customers C and potential customers P, (B) referral customers R. \emph{Bottom row:} (C) optimal control $u_1$, (D) optimal control $u_2$.}
        \label{fig:OCP_beta2}
\end{figure}

\begin{figure}[!htb]
        \centering
        \hspace*{-2cm}
        \begin{subfigure}[b]{0.48\textwidth}\centering
                \includegraphics[scale=0.39]{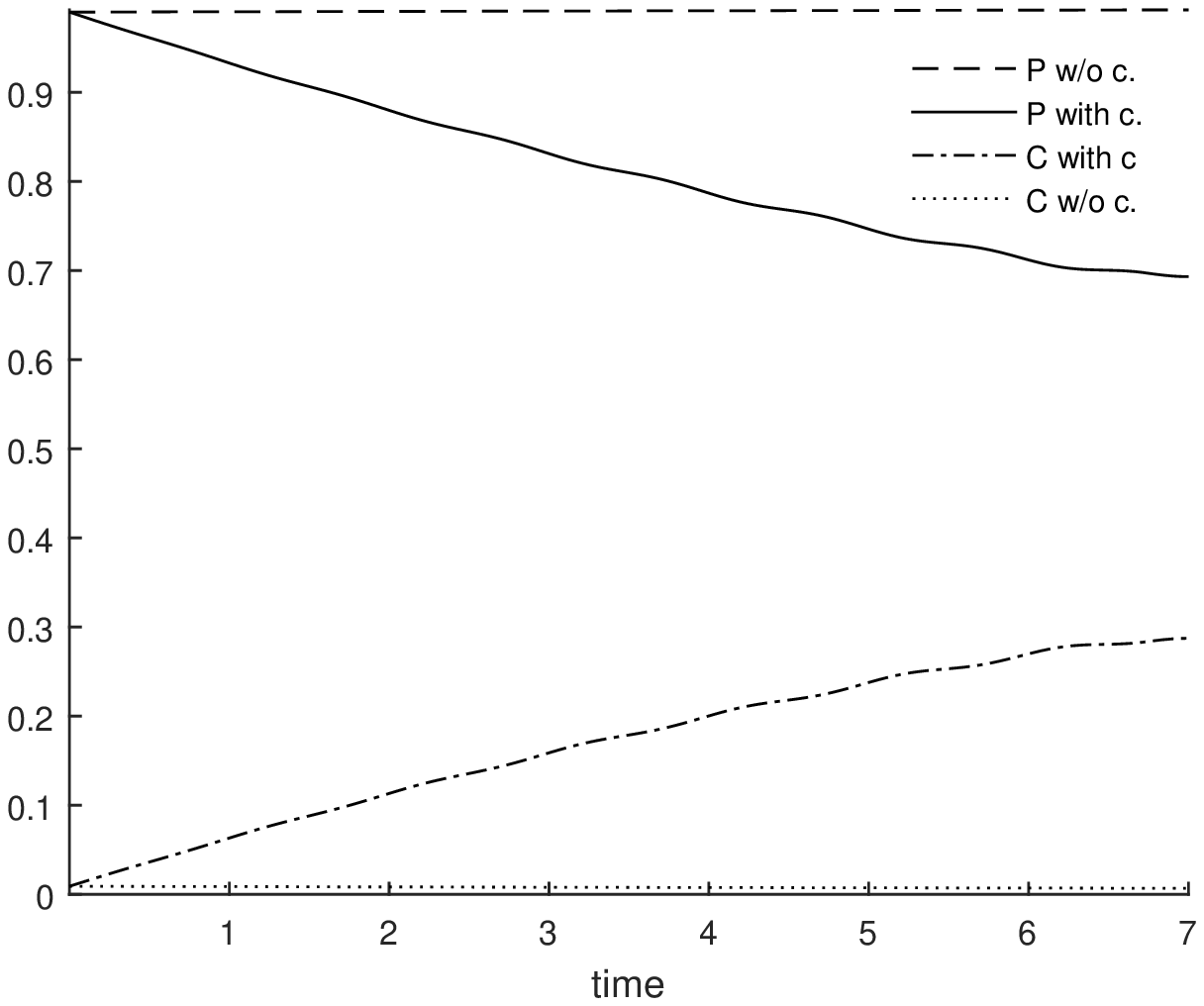}
                \caption{Evolution of $C$ and $P$. }
                \label{fig:PC_beta3}
        \end{subfigure}
        \hspace*{1cm}
        \begin{subfigure}[b]{0.48\textwidth}\centering
                \includegraphics[scale=0.39]{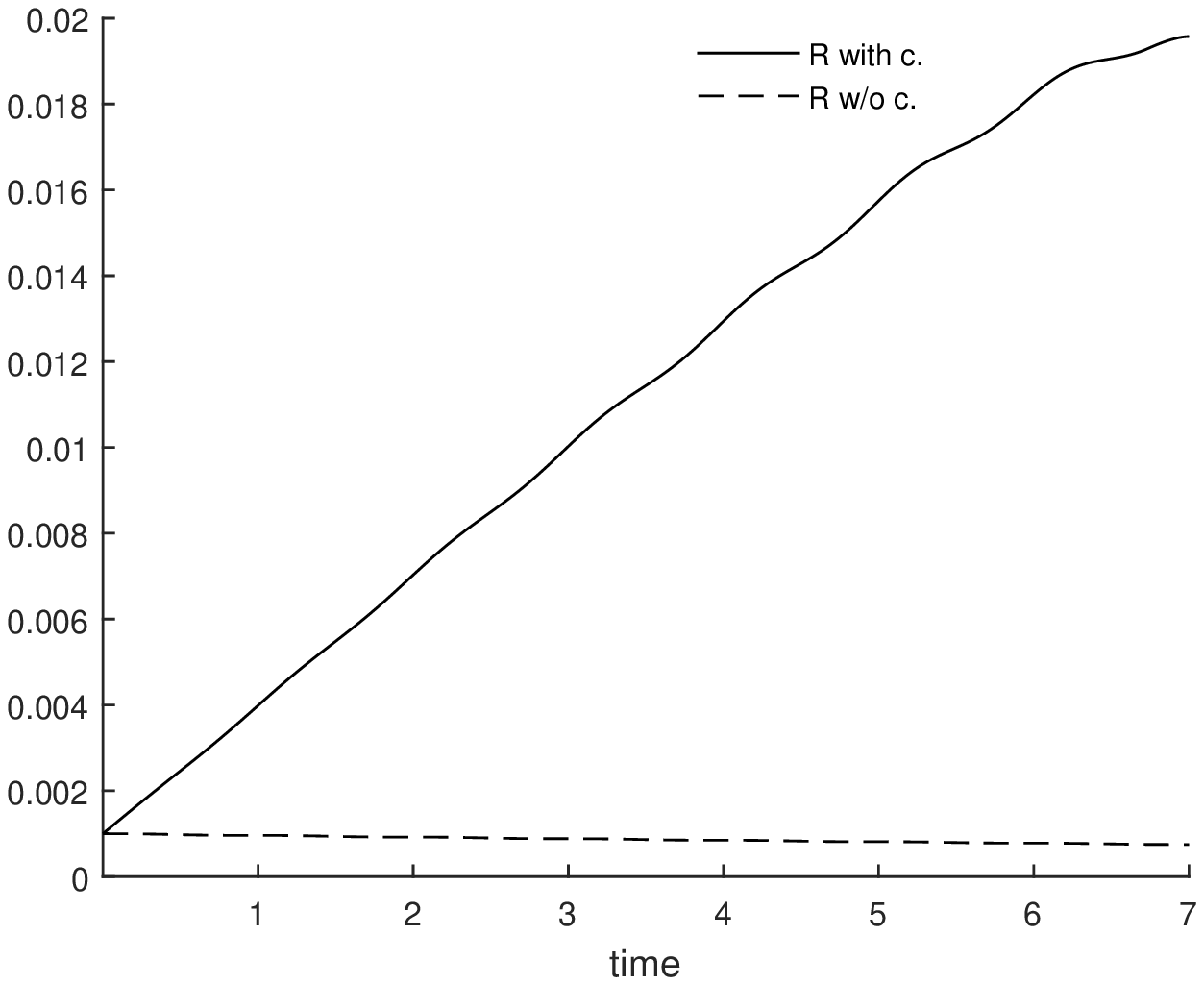}
                \caption{Evolution of $R$. }
                \label{fig:R_beta3}
        \end{subfigure}\\ \centering  \hspace*{-2cm}
        \begin{subfigure}[b]{0.48\textwidth}\centering
                \includegraphics[scale=0.39]{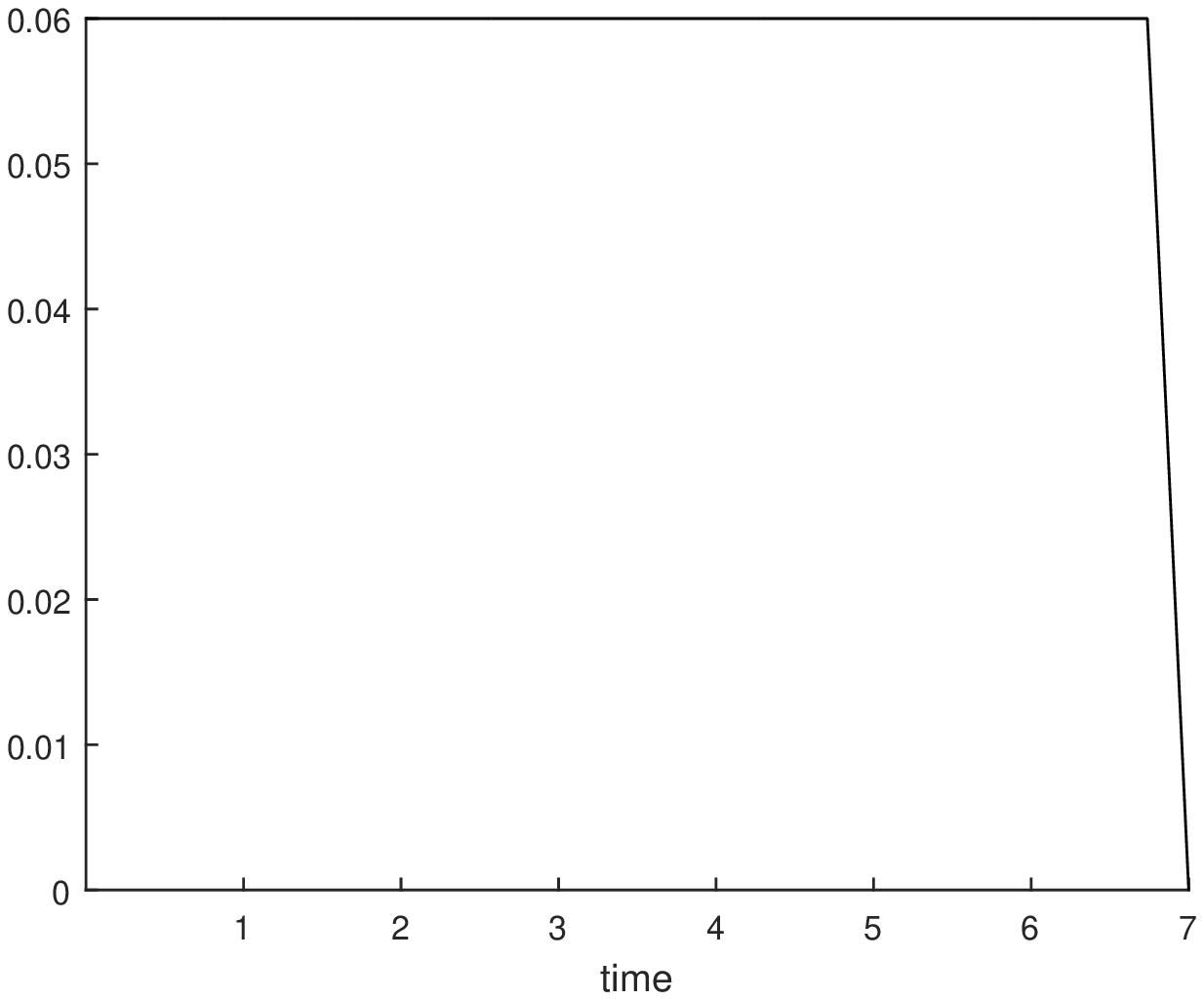}
                \caption{Variation of optimal control $u_1$. }
                \label{fig:u1_beta3}
        \end{subfigure}
        \hspace*{1cm}
        \begin{subfigure}[b]{0.48\textwidth}\centering
                \includegraphics[scale=0.39]{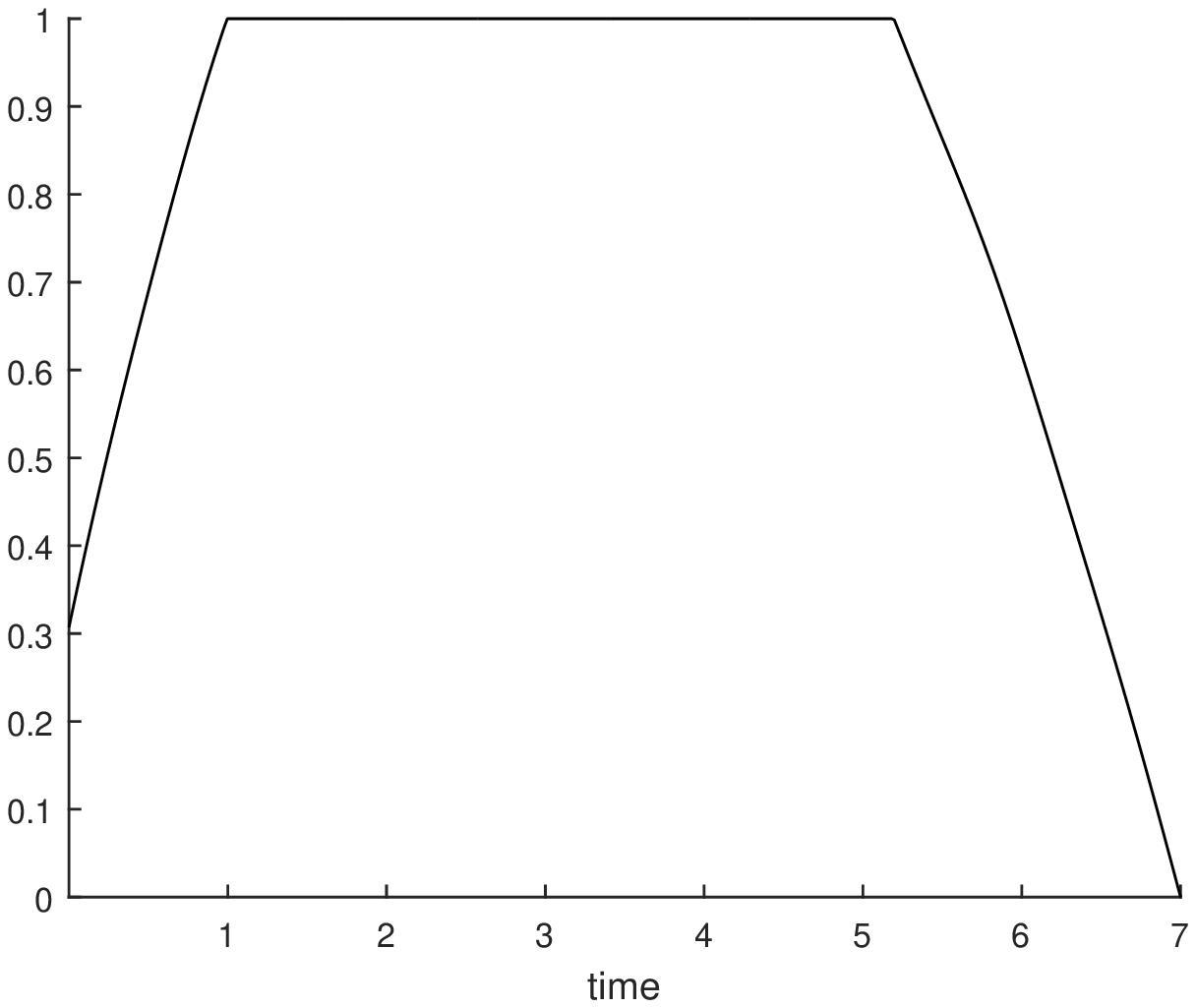}
                \caption{Variation of optimal control $u_2$. }
                \label{fig:u2_beta3}
        \end{subfigure}\\
         \caption{Optimal control and state variables, with control and without control, of the marketing model,  time-varying rate $\beta_3$ and time-varying rate $\gamma_3$. \emph{Top row:} (A) regular customers C and potential customers P, (B) referral customers R. \emph{Bottom row:} (C) optimal control $u_1$, (D) optimal control $u_2$.}
        \label{fig:OCP_beta3}
\end{figure}

\begin{figure}[!htb]
        \centering
        \hspace*{-2cm}
        \begin{subfigure}[b]{0.48\textwidth}\centering
                \includegraphics[scale=0.39]{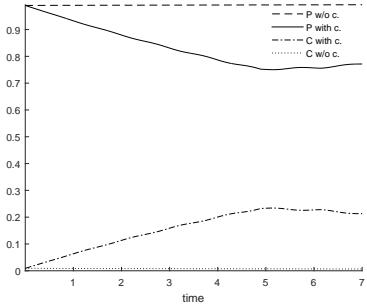}
                \caption{Evolution of $C$ and $P$. }
                \label{fig:PC_beta3_L1}
        \end{subfigure}
        \hspace*{1cm}
        \begin{subfigure}[b]{0.48\textwidth}\centering
                \includegraphics[scale=0.39]{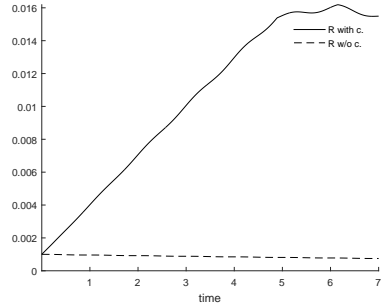}
                \caption{Evolution of $R$. }
                \label{fig:R_beta3_L1}
        \end{subfigure}\\ \centering  \hspace*{-2cm}
        \begin{subfigure}[b]{0.48\textwidth}\centering
                \includegraphics[scale=0.39]{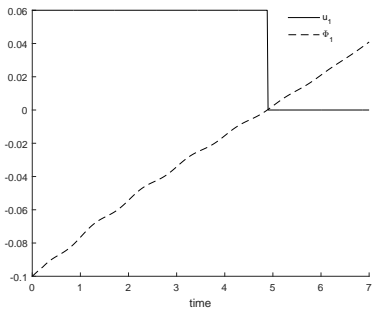}
                \caption{Control $u_1$ and scaled switching function $\Phi_1$. }
                \label{fig:u1_beta3_L1}
        \end{subfigure}
        \hspace*{1cm}
        \begin{subfigure}[b]{0.48\textwidth}\centering
                \includegraphics[scale=0.39]{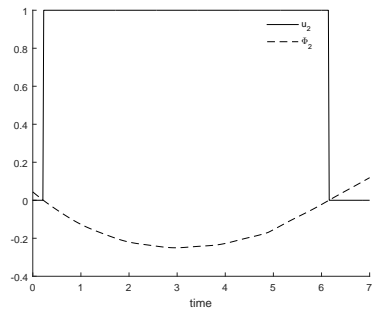}
                \caption{Control $u_2$ and scaled switching function $\Phi_2$. }
                \label{fig:u2_beta3_L1}
        \end{subfigure}
         \caption{Optimal control and state variables, with control and without control, of the marketing model with $L^1$ objective,  time-varying rate $\beta_3$ and time-varying rate $\gamma_3$. \emph{Top row:} (A) regular customers C and potential customers P, (B) referral customers R. \emph{Bottom row:} (C) control $u_1$ and scaled switching function $\Phi_1$, (D) control $u_2$ and scaled switching function $\Phi_2$.}
        \label{fig:OCP_beta3_L1}
\end{figure}

\begin{figure}[!htb]
        \centering
        \hspace*{-1cm}
        \begin{subfigure}[b]{0.48\textwidth}\centering
                \includegraphics[scale=0.39]{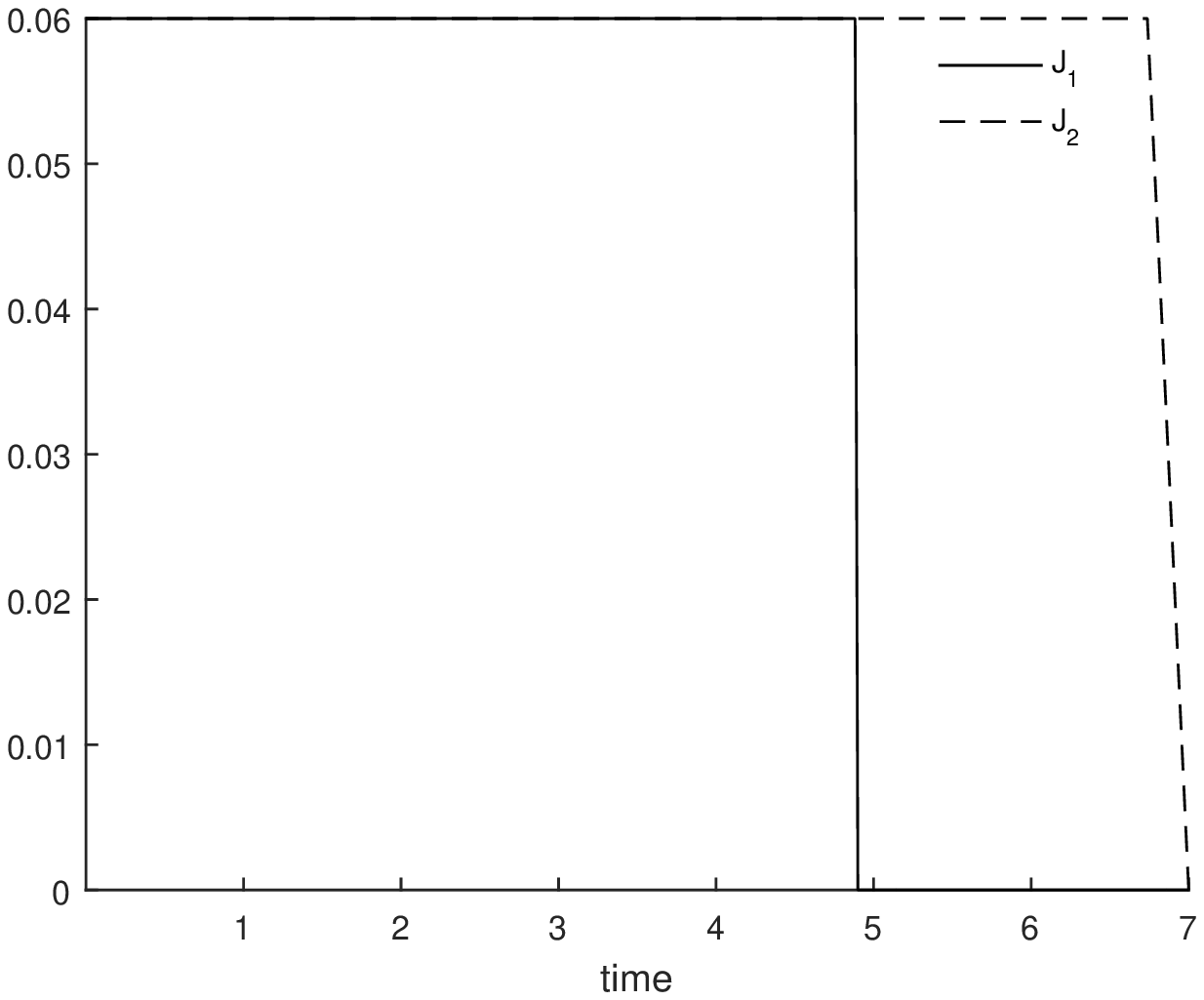}
                \caption{Control $u_1$ for $J_1$ and $J_2$ objectives.}
                \label{fig:compara_u1_beta3}
        \end{subfigure}
        \hspace*{1cm}
        \begin{subfigure}[b]{0.48\textwidth}\centering
                \includegraphics[scale=0.39]{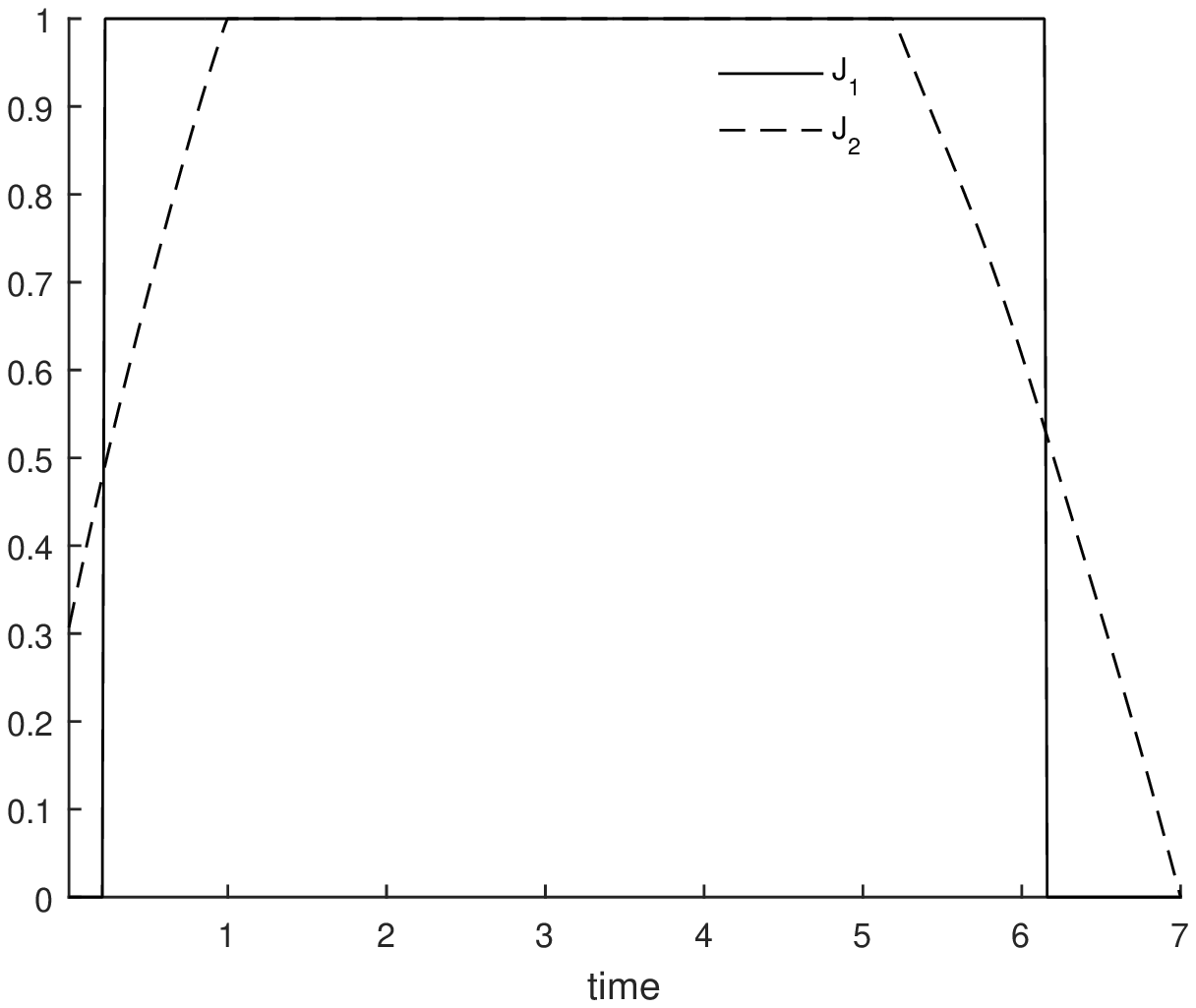}
                 \caption{Control $u_2$ for $J_1$ and $J_2$ objectives.}
                \label{fig:compara_u2_beta3}
        \end{subfigure}\\
         \caption{Comparison of controls $u_1$ and $u_2$ for the $L^1$-type objective ($J_1$) and $L^2$-type objective ($J_2$) with rates $\beta_3$ and $\gamma_3$.}
        \label{fig:compara_controlos_beta3}
\end{figure}

The  optimal control problem is numerically solved using a Runge-Kutta fourth order iterative method. First we solve the system (\ref{eq:modelo})-(\ref{eq:cond_ini}), by the forward Runge-Kutta fourth order procedure, and obtain the values of the state variables ($C$, $R$ and $P$). Using those values, then we solve the system (\ref{eq:Pontryagin-CPR-Mayer-1}) with the transversality conditions (\ref{eq:Pontryagin-CPR-Mayer-5}), by  backward fourth order Runge-Kutta procedure, and obtain the values of the co-state variables.  The controls are updated by a convex combination of the
previous values and the new values computed according with (\ref{eq:Pontryagin-CPR-Mayer-6})-(\ref{eq:Pontryagin-CPR-Mayer-7}). The iteration is stopped when the values of the unknowns at the earlier
iteration are very close to the ones at the current iteration.

In what follows, we assume that the maximum rate of direct recruitment of individuals from the population is ${u_1}_{\max}=0.06$ (cf. \cite{Kandhway2014}). The word-of-mouth control is potentiated by the referrals and is expected that all of them may act as spreaders, so ${u_2}_{\max}=1.0$. The terminal time is $t_f=7$ time units and the remaining  parameters are fixed according to Table \ref{tab:ref}. The initial conditions are the following:
$$C_0=0.009, R_0=0.001\mbox{ and } P_0=0.99$$ 
We consider that the weight values are $\kappa_1=1$, $\kappa_2=1.5$ and $\kappa_3=0.01$.

 We study next the optimal control problem with time dependent rates  $\beta(t)$ and $\gamma(t)$. After, in Section \ref{sub:effectiveness}, we analyse the effect of some parameters ($\gamma$, $\kappa_2$, $\beta$ and $t_f$) on the cost functional $J$ of the proposed model.

\subsection{Variable recruitment rate  $\beta(t)$ and variable defection rate $\gamma(t)$.}

  Inspired in \cite{Kandhway2014}, to model the varying interest of a population in recruit new customers during the campaign duration, we consider three different functions $\beta_1(t)$, $\beta_2(t)$ and  $\beta_3(t)$. They model the cases of increasing, decreasing and fluctuating  interest as the action of the  referral customers develops, respectively. The respective functions, exhibited in Figure \ref{fig:betas:3}, are defined as:
\begin{align*}
&\beta_1(t) = 0.01+\frac{0.99}{1+e^{-2t+8}},\\[1mm]
&\beta_2(t) = 0.01 + 0.99\left(1-\frac{1}{1+e^{-2t+6}}\right),\\[1mm]
&\beta_3(t)= 0.01+0.49\left(1 - \cos( 2 \pi t + 0.26)\right),
\end{align*}
The increasing recruiting rate, $\beta_1(t)$, may represent the increasing interest of people by election candidates as we approach the polling date. The decreasing recruiting rate, $\beta_2(t)$, may represent gradual loss of interest of people in some product after its release (e.g a newly launched smartphone). Fluctuating recruiting rate, $\beta_3(t)$, may represent changes in demand of a product with time (e.g seasonal products that have great demand in a given season  but little demand during the rest of the year).

Three distinct scenarios are also proposed to the defection rate to complement the three cases proposed to the recruitment rate. In the first scenario we propose that the defection rate is invariant. In second we suggest that the decreasing recruiting rate is followed by an increasing defection rate. In last scenario we propose that the oscillating interest is accompanied by an also oscillating defection rate. The $\gamma_i\parc{t}$  functions for $i=1,2,3$, exhibited in Figure \ref{fig:gamas:3},  are as follows, \vspace*{-0.2cm}
\begin{align*}
&\gamma_1(t)=\gamma_0,\\
&\gamma_2(t) = 0.01+\frac{0.18}{1+e^{-2t+7}},\\
&\gamma_3(t)= \gamma_0\left(1- 0.9 \cos( 2 \pi t +0.26)\right),\label{gamma_fun} \end{align*}~where $\gamma_0=0.10.$

In the case of the increasing interest of a population in recruiting new customers, during the campaign duration (rates $\beta_1$ and $\gamma_1$),  the solution for the optimal control problem and solution to the no control problem is illustrated in Figure \ref{fig:OCP_beta1}. In what concerns the optimal solution, the number of referral customers grows vigorously and reaches its maximum,  close to 0.02, almost at the terminal time. We also notice that the number of customers, referral and regular, evidence a very light decrease when approaching the end of time interval.  On the other hand, the number of referral and regular costumers, of the no control solution, are lower than the ones of the optimal solution. In the solution of the no control problem, of the two remaining scenarios, the number of costumers is also lower than the one of optimal solution (see \Cref{fig:OCP_beta2,fig:OCP_beta3,fig:OCP_beta3_L1}). 

In Figures \ref{fig:u1_beta1}, \ref{fig:u2_beta1}, while the first control, $u_1$, is maximum in almost all the time interval, second control, $u_2$, is maximum on a central part of the interval since $t_1$ (close to 1) up to $t_2$ (between 5 and 6). The controls we obtain for the following two cases, displayed in Figures \ref{fig:u1_beta2},\ref{fig:u2_beta2},\ref{fig:u1_beta3},\ref{fig:u2_beta3}, are analogous to these ones.

In the case of the decreasing interest of a population in recruiting new customers, during the campaign duration (rates $\beta_2$ and $\gamma_2$),  the solution for the optimal control problem and solution of no control problem are exhibited in Figure \ref{fig:OCP_beta2}. Relatively to the optimal solution, the evolution of the number of referrals also grows vigorously in the beginning, but slow down in the second half of time interval. When approaching terminal time, the number of customers, referral and regular, exhibit a reduction  bigger than the preceding case. This behaviour  is motivated by the recruitment rate $\beta_2$.


In the case of the periodic interest of a population in recruiting new customers, during the campaign duration,  the solution for the optimal control problem  with the $L^2$ objective and the solution of the no control problem are displayed in Figure \ref{fig:OCP_beta3}. The periodic nature of the parameters $\beta_3$ and $\gamma_3$ influences the evolution of the three state variables. Relatively to the optimal solution, the variation of the number of customers, referral and regular, is, in general, similar to the first case.

The optimal control problem with $L^1$ objective functional, presented in \ref{appendixA}, was also considered in case where the interest in recruiting new customers is periodic (third case). The optimal solution,  obtained analogously with the  Runge-Kutta scheme presented above, and the no control solution are presented in Figure \ref{fig:OCP_beta3_L1}. Relatively to such optimal solution, the customers, referral and regular, stop growing when the first control, $u_1$, becomes inactive and their maximums are smaller than those that were obtained with the quadratic objective. It can also be observed that the switching functions satisfy the strict bang-bang property (cf. \cite{Maurer2012}) associated to the Pontryagin Maximum Principle. 

Figure \ref{fig:compara_controlos_beta3}  compares  the optimal controls for the linear functional, $J_1$, with the quadratic functional, $J_2$. The first control variable $u_1$ differ on a terminal interval $t_i<t<t_f$ where we see that the $J_1$ control is inactive while the $J_2$ control is active (maximum). The second control $u_2$ shows also differences for the two functionals at beginning and at end of time interval. In Figure \ref{fig:compara_u2_beta3} we can see that these differences are somehow compensated. Like in other analogous works   where the upper bound equals the value one and the graphics of control solutions are similar (e.g \cite{preprint:Delfim2016}), the optimal state variables of the two functionals are almost identical.   

\subsection{Comparison of optimal control with simpler controls.}\label{sub:effectiveness}

The recruiting rate, $\beta$, and the defection rate, $\gamma$, are constant in this section. The goal of this section is to compare the effectiveness of optimal control strategy with other simpler control strategies that do not require any optimization technique. 

As in \cite{Kandhway2014}, we compare the optimal control problem with three more problems. Namely: 1) the problem without control (controls are zero); 2) problem where controls are constant with $u_1'(t)=(1-\alpha_1){u_1}_{\max}/2$ and  $u_2'(t)=\alpha_2{u_2}_{\max}/2$\footnotemark[1] ; 3) problem with heuristic controls, know as \emph{follow $P_{nc}(t)$, $P_{nc}(t)R_{nc}(t)$} (see \cite{Kandhway2014}), where controls are $u_1''(t)=(1-\alpha_1){u_1}_{\max} P_{nc}(t)$ and  $u_2''(t)=\alpha_2{u_2}_{\max}P_{nc}(t)R_{nc}(t)$\footnotemark[1],  being $P_{nc}(t)$ and $R_{nc}(t)$ the fractions of potential customers and referral customers, respectively, when no control is applied.
\footnotetext[1]{Since $u_2(t)$, in most cases, has rather low values, we multiply ${u_2}_{\max}$ by  an small constant, $\alpha_2$.}

 In order to compare the optimal strategy with the remaining strategies, using  ranges of values for parameters $\gamma$, $\beta$, $t_f$ and $\kappa_2$ similar to \cite{Kandhway2014}, we use by default, in what follows, the weight values: $\kappa_1=1/t_f$, $\kappa_2=15$ and $\kappa_3=1$.


\begin{figure}[!htb]
        \centering
       \hspace*{-1cm}
        \begin{subfigure}[b]{0.48\textwidth}\centering
                \includegraphics[scale=0.39]{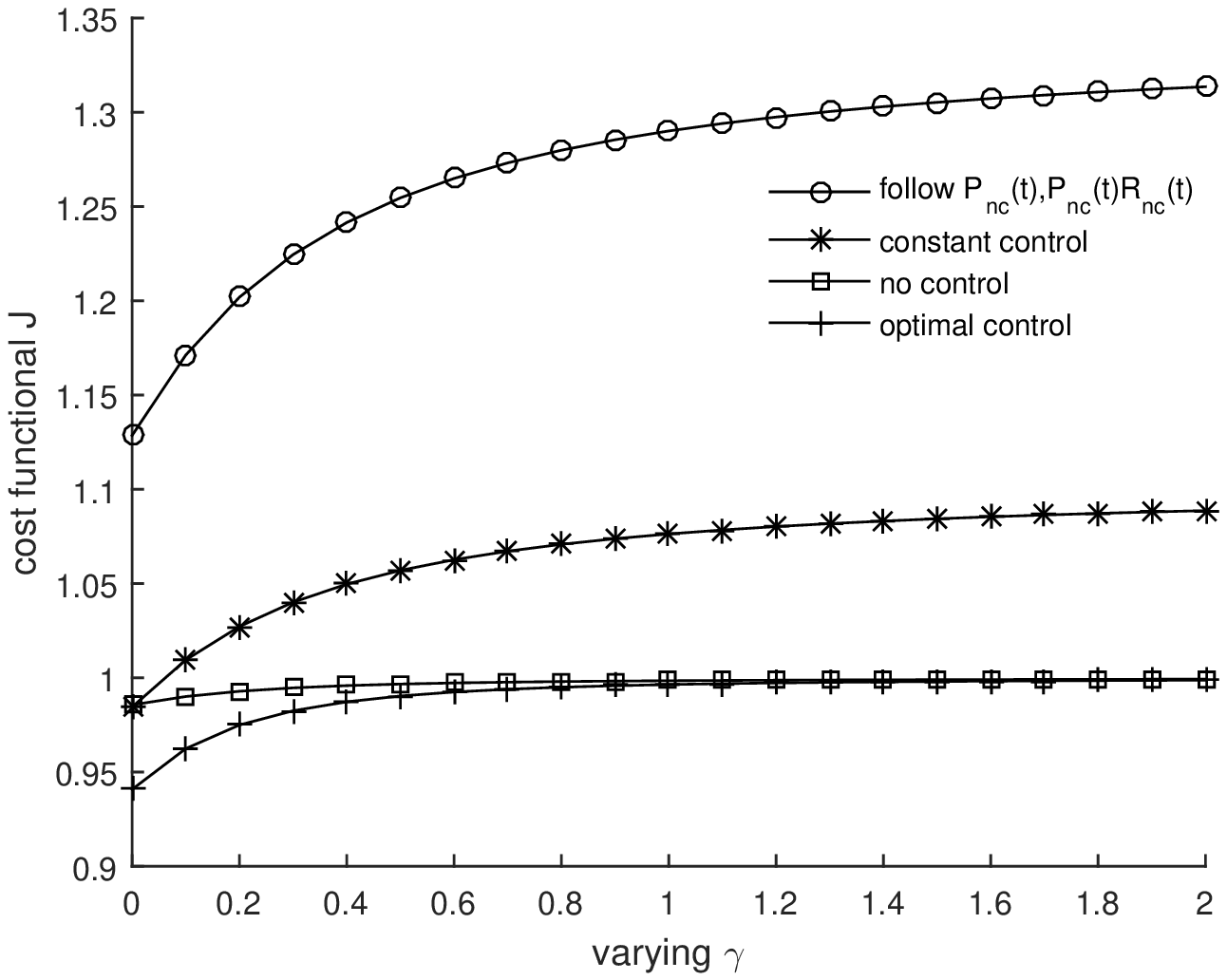}
                \caption{}
                \label{fig:effectiveness_gamma}
        \end{subfigure}
        \hspace*{1cm}
        \begin{subfigure}[b]{0.48\textwidth}\centering
                \includegraphics[scale=0.39]{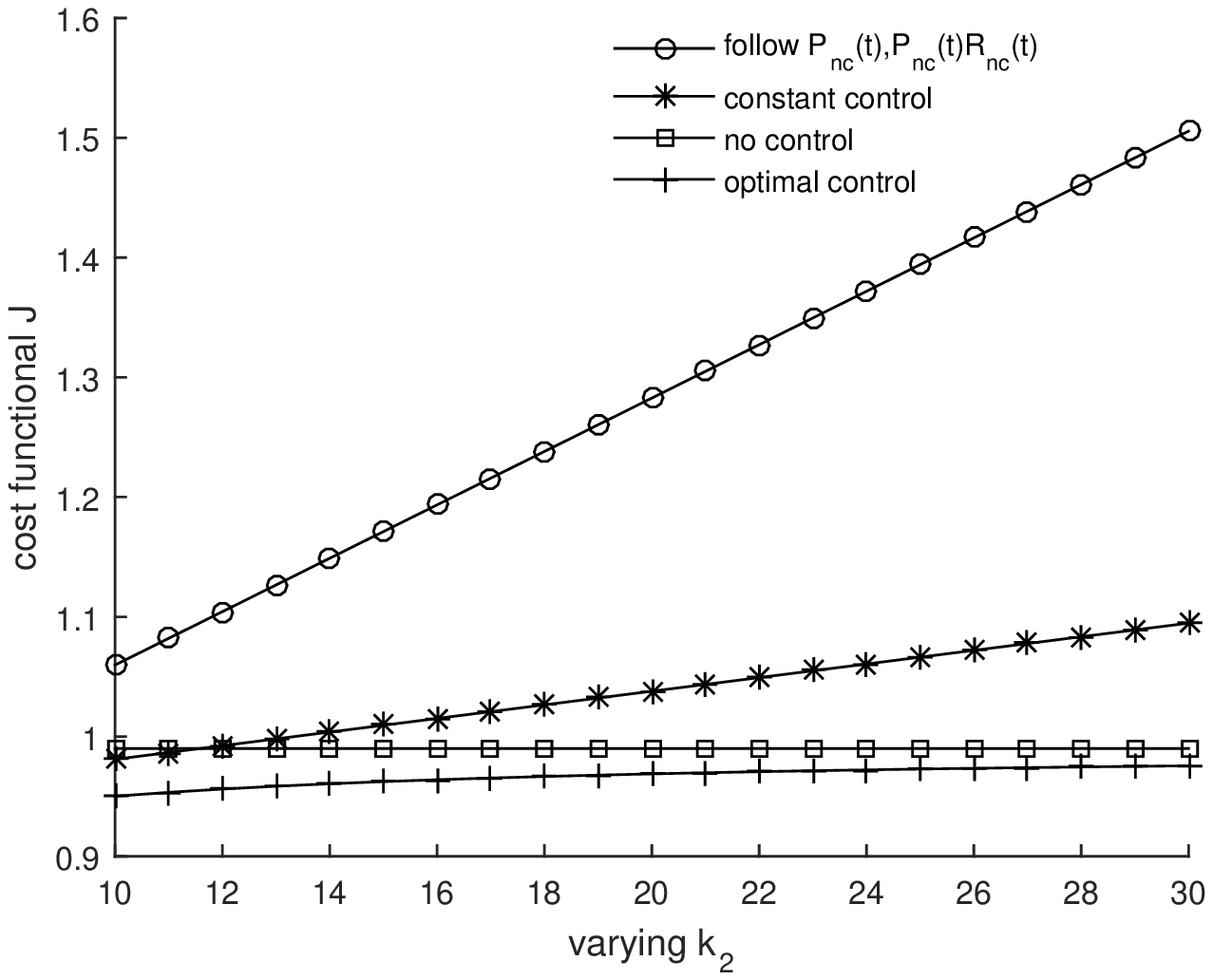}
                 \caption{}
                \label{fig:effectiveness_k2}
        \end{subfigure}\\
         \caption{Evolution of the value of the objective functional $J$ with: (A) variation of the defection rate $\gamma$ with $\kappa_2=15$, (B) variation of the weight parameter $\kappa_2$ with $\gamma=0.1$. Parameter values: $\beta=1$, $t_f=7$.}
        \label{fig:effectiveness}
\end{figure}

In Figure \ref{fig:effectiveness_gamma} we display the evolution of the cost function, for  the four problems (or strategies) above mentioned, with variation of the defection rate $\gamma$. The cost functional $J$ for the constant control strategy and \emph{follow $P_{nc}(t)$, $P_{nc}(t)R_{nc}(t)$} strategy is bigger than $J$ for no control. On the other hand, $J$ for the optimal control strategy is smaller than $J$ for no control strategy for $\gamma$ lower than 1.0, but these $J$s coincide when $\gamma$ is greater or equal to 1.0. We may conclude that when $\gamma\geqslant 1$ the optimal control is ineffective and there is no need of any marketing campaign.

In Figure \ref{fig:effectiveness_k2} we display the evolution of the cost function with variation of the weight parameter $\kappa_2$. The cost functional $J$ for the constant control strategy and \emph{follow $P_{nc}(t)$, $P_{nc}(t)R_{nc}(t)$} strategy are bigger than the $J$ for the optimal control. We can also see that $J$ for the no control strategy is bigger than $J$ for the optimal control strategy when $\kappa_2$ is small, but these strategies tend to have the same values  of $J$ as we consider bigger values for $\kappa_2$. 

We also compared the evolution of the cost function of the optimal control model with the other three strategies for the variation of $\beta$ ($\beta\in[0,3]$) and $t_f$ ($t_f\in[4,14]$). The resulting figures are not displayed because, in these cases, the optimal control strategy was the one with smaller values of cost function and no tendency to approach to one of the other strategies was exhibited. Hence, in these cases the optimal control is recommended.


\section{Conclusions}
In this paper we have considered an optimal control problem for a nonlinear system of ordinary differential equations that describes the evolution of the number of regular customers and referral customers in some firm. The aim is to study, considering several types of behaviour for the population, the best marketing strategy in the decision to invest in referrals programs.

 The existence and uniqueness of optimal solutions was established for an $L^2$ cost functional model. Some simulation results of such model were presented  and compared with  the ones obtained for the model with an $L^1$ cost functional. The optimal solutions for the problem with linear lagrangian  are of bang-bang type.

 While performing the numerical simulations, we have noticed that, for some values of the cost weights $\kappa_1$, $\kappa_2$ and $\kappa_3$, the solutions for the quadratic objective model are slightly better than the ones for the linear objective model.  Nevertheless, the strategy obtained for  the linear objective model is easier to implement, since at each time interval the possible actions are taken from a finite set of possibilities, and thus may be more appealing to the marketing managers. 

For the autonomous case of quadratic cost functional model, we have shown the  effectiveness of the optimal control strategy over the constant control strategy, a heuristic control strategy and the no control. 

\appendix
\section{Proof of Theorem \ref{teo:uniqueness}.}\label{appendix_proof}
\begin{proof}
We assume that we have two optimality systems corresponding to trajectories and state equations $(R,C,P)$, $(p_1,p_2,p_3)$ and $(\bar R,\bar C,\bar P)$, $(\bar p_1,\bar p_2,\bar p_3)$ and we will show that the two coincide in some small interval. Consider the change of variables
$$R(t)=e^{\theta t}r(t), \quad C(t)=e^{\theta t}c(t), \quad P(t)=e^{\theta t}q(t)$$
and
$$p_1(t)=e^{-\theta t}\phi_1(t), \quad p_2(t)=e^{-\theta t}\phi_2(t), \quad p_3(t)=e^{-\theta t}\phi_3(t).$$

Recall that $N_0=R(t)+C(t)+P(t)$ is constant and that the region \linebreak $\{(P,C,R)\in (\R^+_0)^3: P+C+R=N_0\}$ is forward invariant.

By the first equation in~\eqref{eq:modelo} we get
$$\theta e^{\theta t}r+e^{\theta t} \dot{r}=-\lambda_2\e^{\theta t}r+\lambda_1\e^{\theta t}c-\beta_1\e^{\theta t}r +\alpha u_1\e^{\theta t}q+\alpha u_2\e^{2\theta t}rq/N$$
and thus
$$\theta r+ \dot{r}=-\lambda_2r+\lambda_1c-\beta_1r +\alpha u_1q+\alpha \e^{\theta t}u_2rq/N.$$
Subtracting the corresponding barred equation from the above equation we get
$$\theta (r-\bar r)+ \dot{r}-\dot{\bar r}=-(\lambda_2+\beta_1)(r-\bar r)+\lambda_1(c-\bar c)+\alpha (u_1q- \bar u_1\bar q)+\alpha \e^{\theta t}(u_2 r q - \bar u_2 \bar r \bar q)/N.$$
Multiplying by $(r-\bar{r})$, integrating from $0$ to $T$ and noting that $r(0)=\bar{r}(0)$ we have
\begin{equation}\notag 
\begin{split}
&\frac{1}{2}(r(T)-\bar r(T))^2 + \theta \int_0^T (r-\bar r)^2 dt \\
& =  - (\lambda_2+\beta_1) \int_0^T (r-\bar r)^2 dt + \lambda_1 \int_0^T (c-\bar c)(r-\bar r) dt \\
& \quad + \alpha  \int_0^T (u_1q- \bar u_1\bar q)(r-\bar r) dt + \frac{\alpha \e^{\theta T}}{N} \int_0^T (u_2 r q - \bar u_2 \bar r \bar q)(r-\bar r) dt
\end{split}
\end{equation}
and there are $C_1,C_2>0$ such that
\begin{equation}\label{eq:estimativa-necessaria-r}
\begin{split}
&\frac{1}{2}(r(T)-\bar r(T))^2 + \theta \int_0^T (r-\bar r)^2 dt \\
& =(\lambda_1/2+\alpha C_1+\alpha C_2 \e^{\theta T}/N-\lambda_2-\beta_1) \int_0^T (r-\bar r)^2 dt + \lambda_1/2 \int_0^T (c-\bar c)^2 dt \\
& \quad + \alpha C_1  \int_0^T (u_1- \bar u_1)^2 dt + \frac{\alpha C_1+\alpha C_2 \e^{\theta T}}{N} \int_0^T(q-\bar q)^2 dt \\
& \quad  + \frac{\alpha^2 C_1 C_2 \e^{\theta T}}{N} \int_0^T (u_2 - \bar u_2)^2 dt
\end{split}
\end{equation}

By the second equation in~\eqref{eq:modelo} we get
$$\theta e^{\theta t}c+e^{\theta t} \dot{c}=-\lambda_1\e^{\theta t}c+\lambda_2\e^{\theta t}r-\beta_2\e^{\theta t}c +(1-\alpha) u_2\e^{2\theta t}rq/N+(1-\alpha) u_1\e^{\theta t}q$$
and thus
$$\theta c+\dot{c}=-\lambda_1c+\lambda_2r-\beta_2c +(1-\alpha) u_2\e^{\theta t}rq/N+(1-\alpha) u_1q.$$
Subtracting the corresponding barred equation from the above equation we get
\[
\begin{split}
\theta (c-\bar c)+ \dot{c}-\dot{\bar c}
& = -(\lambda_1+\beta_2)(c-\bar c)+\lambda_2(r-\bar r) \\
& \quad +(1-\alpha)\e^{\theta t}(u_2rq-\bar u_2\bar r\bar q)/N+(1-\alpha) (u_1q-\bar u_1\bar q).
\end{split}
\]
Multiplying by $(c-\bar{c})$, integrating from $0$ to $T$ and noting that $c(0)=\bar{c}(0)$ we have
\begin{equation}\notag 
\begin{split}
&\frac{1}{2}(c(T)-\bar c(T))^2 + \theta \int_0^T (c-\bar c)^2 dt \\
& =  - (\lambda_1+\beta_2) \int_0^T (c-\bar c)^2 dt + \lambda_2 \int_0^T (c-\bar c)(r-\bar r) dt \\
& \quad + \frac{(1-\alpha) \e^{\theta T}}{N} \int_0^T (u_2rq-\bar u_2\bar r\bar q)(c-\bar c) dt + (1-\alpha) \int_0^T (u_1 q - \bar u_1 \bar q)(c-\bar c) dt
\end{split}
\end{equation}
and there are $C_3,C_4>0$ such that
\begin{equation}\label{eq:estimativa-necessaria-c}
\begin{split}
&\frac{1}{2}(c(T)-\bar c(T))^2 + \theta \int_0^T (c-\bar c)^2 dt \\
& = (\lambda_2/2+(1-\alpha) C_3\e^{\theta T}/N+ (1-\alpha)C_4-\lambda_1-\beta_2) \int_0^T (c-\bar c)^2 dt\\
& \quad + (\lambda_2/2+(1-\alpha) C_3\e^{\theta T}/N) \int_0^T (r-\bar r)^2 dt\\
& \quad + (1-\alpha) (C_3\e^{\theta T}/N+C_4) \int_0^T (q-\bar q)^2 dt \\
& \quad + \frac{(1-\alpha) C_3\e^{\theta T}}{N} \int_0^T (u_2-\bar u_2)^2 dt + (1-\alpha) C_4 \int_0^T (u_1- \bar u_1)^2 dt
\end{split}
\end{equation}

By the third equation in~\eqref{eq:modelo} we get
$$\theta e^{\theta t}q+e^{\theta t} \dot{q}=-u_2\e^{2\theta t}qr/N-u_1\e^{\theta t}q+\beta_1\e^{\theta t}r + \beta_2\e^{\theta t}c $$
and thus
$$\theta q+\dot{q}=-u_2\e^{\theta t}qr/N-u_1q+\beta_1r + \beta_2c.$$
Subtracting the corresponding barred equation from the above equation we get
\[
\begin{split}
\theta (q-\bar q)+ \dot{q}-\dot{\bar q}
& = - \e^{\theta t} (u_2qr-\bar u_2\bar q\bar r)/N-(u_1q-\bar u_1\bar q)+\beta_1(r-\bar r) + \beta_2(c-\bar c).
\end{split}
\]
Multiplying by $(q-\bar{q})$, integrating from $0$ to $T$ and noting that $q(0)=\bar{q}(0)$ we have
\begin{equation}\notag 
\begin{split}
&\frac{1}{2}(q(T)-\bar q(T))^2 + \theta \int_0^T (q-\bar q)^2 dt \\
& =  - \frac{\e^{\theta T}}{N} \int_0^T (u_2 qr-\bar u_2\bar q\bar r)(q-\bar q) dt -\int_0^T (u_1 q-\bar u_1\bar q)^2 dt \\
& \quad + \beta_1 \int_0^T (r-\bar r)(q-\bar q) dt + \beta_2 \int_0^T (c - \bar c)(q- \bar q) dt
\end{split}
\end{equation}
and there are $C_5,C_6>0$ such that
\begin{equation}\label{eq:estimativa-necessaria-q}
\begin{split}
&\frac{1}{2}(q(T)-\bar q(T))^2 + \alpha \int_0^T (q-\bar q)^2 dt \\
& =  (\beta_1/2- \e^{\theta T} C_5/N) \int_0^T (r-\bar r)^2 dt + \beta_2/2 \int_0^T (c - \bar c)^2 dt\\
& \quad +(\beta_1/2+\beta_2/2-\e^{\theta T} C_5/N-C_6) \int_0^T (q-\bar q)^2 dt \\
& \quad - \frac{\e^{\theta T} C_5}{N} \int_0^T (u_2-\bar u_2)^2 dt- C_6 \int_0^T (u_1-\bar u_1)^2 dt
\end{split}
\end{equation}
To obtain a bound for the controls we use the conditions given by~\eqref{eq:Pontryagin-CPR-Mayer-6} and~\eqref{eq:Pontryagin-CPR-Mayer-7}. We have
\begin{equation}\label{m1-bar-m1}
\begin{split}
& (u_1-\bar u_1)^2 \\
& \le \left[(p_3-p_1\alpha-p_2(1-\alpha))P/(2\kappa_2)-(\bar p_3-\bar p_1\alpha-\bar p_2(1-\alpha))\bar P/(2\kappa_2)\right]^2\\
& \le (C_7+\tilde C_7\e^{\theta T})[(p-\bar p)^2+(\phi_1-\bar \phi_1)^2+(\phi_2-\bar \phi_2)^2+(\phi_3-\bar \phi_3)^2]
\end{split}
\end{equation}
and
\begin{equation}\label{m2-bar-m2}
\begin{split}
& (u_2-\bar u_2)^2 \\
& \le \left[(p_3-p_1\alpha-p_2(1-\alpha))PR/(2\kappa_3N)-(\bar p_3-\bar p_1\alpha-\bar p_2(1-\alpha))\bar P\bar R/(2\kappa_3N)\right]^2\\
& \le (C_8+\tilde C_8\e^{\theta T})[(p-\bar p)^2+(r-\bar r)^2+(\phi_1-\bar \phi_1)^2+(\phi_2-\bar \phi_2)^2+(\phi_3-\bar \phi_3)^2].
\end{split}
\end{equation}
Next, using~\eqref{m1-bar-m1} and~\eqref{m2-bar-m2}, we obtain
\begin{equation}\label{eq:estimativa-necessaria-phi-1}
\begin{split}
&\frac{1}{2}(\phi_1(0)-\bar \phi_1(0))^2 + \theta \int_0^T (\phi_1-\bar \phi_1)^2 dt \\
& \le (C_9+\tilde C_9\e^{\theta T})\int_0^T (u_2-\bar u_2)^2\\ & \quad +(\phi_1-\bar\phi_1)^2+(\phi_2-\bar\phi_2)^2+(\phi_3-\bar\phi_3)^2
+(c-\bar c)^2+(p-\bar p)^2 \dt \\
& \le (C_{10}+\tilde C_{10}\e^{\theta T})  \int_0^T(\phi_1-\bar\phi_1)^2+(\phi_2-\bar\phi_2)^2\\
& \quad +(\phi_3-\bar\phi_3)^2
+(c-\bar c)^2+(p-\bar p)^2+(r-\bar r)^2 \dt,
\end{split}
\end{equation}
\begin{equation}\label{eq:estimativa-necessaria-phi-2}
\begin{split}
&\frac{1}{2}(\phi_2(0)-\bar \phi_2(0))^2 + \theta \int_0^T (\phi_2-\bar \phi_2)^2 dt \\
& \le (C_{11}+\tilde C_{11}\e^{\theta T}) \int_0^T(u_2-\bar u_2)^2+(\phi_1-\bar\phi_1)^2\\
& \quad +(\phi_2-\bar\phi_2)^2+(\phi_3-\bar\phi_3)^2
+(r-\bar r)^2+(p-\bar p)^2 \dt \\
& \le (C_{12}+\tilde C_{12}\e^{\theta T}) \int_0^T(\phi_1-\bar\phi_1)^2+(\phi_2-\bar\phi_2)^2+(\phi_3-\bar\phi_3)^2\\
& \quad +(p-\bar p)^2+(r-\bar r)^2 \dt
\end{split}
\end{equation}
and
\begin{equation}\label{eq:estimativa-necessaria-phi-3}
\begin{split}
&\frac{1}{2}(\phi_3(0)-\bar \phi_3(0))^2 + \theta \int_0^T (\phi_3-\bar \phi_3)^2 dt \\
& \le (C_{13}+\tilde C_{13}\e^{\theta T}) \int_0^T (u_1-\bar u_1)^2+(u_2-\bar u_2)^2+(\phi_1-\bar\phi_1)^2 \dt\\
& \quad +(\phi_2-\bar\phi_2)^2 +(\phi_3-\bar\phi_3)^2+(r-\bar r)^2+(c-\bar c)^2 \dt \\
& \le (C_{14}+\tilde C_{14}\e^{\theta T}) \int_0^T (\phi_1-\bar\phi_1)^2+(\phi_2-\bar\phi_2)^2\\
& \quad +(\phi_3-\bar\phi_3)^2
+(p-\bar p)^2+(r-\bar r)^2+(c - \bar c)^2 \dt.
\end{split}
\end{equation}

Let
$$\Psi(t)=(r(t)-\bar r(t))^2+(c(t)-\bar c(t))^2+(q(t)-\bar q(t))^2$$
and
$$\Phi(t)=(\phi_1(t)-\bar \phi_1(t))^2+(\phi_2(t)-\bar \phi_2(t))^2+(\phi_3(t)-\bar \phi_3(t))^2.$$
Adding equations ~\eqref{eq:estimativa-necessaria-r},~\eqref{eq:estimativa-necessaria-c}, ~\eqref{eq:estimativa-necessaria-q},~\eqref{eq:estimativa-necessaria-phi-1},~\eqref{eq:estimativa-necessaria-phi-2} and~\eqref{eq:estimativa-necessaria-phi-3} we obtain for the sum of left-hand sides
$$\frac{1}{2}\Psi(T)+\frac{1}{2}\Phi(0)+\theta\int_0^T\Psi(T)+\Phi(T)dt$$
and thus
\[
\begin{split}
&\frac{1}{2}[\Psi(T)+\Phi(0)]+\alpha\int_0^T\Psi(T)+\Phi(T)dt \\
& \le \tilde{C}\int_0^T\Psi(T)+\Phi(T)dt+\hat{C}e^{\alpha T}\int_0^T\Psi(T)+\Phi(T)dt\\
\end{split}
\]
witch is equivalent to
\begin{equation}\label{ineq}
\begin{split}
&\frac{1}{2}[\Psi(T)+\Phi(0)]+(\theta-\tilde{C}-\hat{C}e^{\theta T})\int_0^T\Psi(T)+\Phi(T)dt \le 0.
\end{split}
\end{equation}

We now choose $\theta$ so that
$$\theta>\tilde{C}+\hat{C}$$
and note that $\frac{\theta-\tilde{C}}{\hat{C}}>1$. Subsequently, we choose $T$ such that
$$T<\frac{1}{\theta}\ln\left(\frac{\theta-\tilde{C}}{\hat{C}}\right).$$

Then,
$$\theta T<\ln\left(\frac{\theta-\tilde{C}}{\hat{C}}\right) \quad \Rightarrow \quad e^{\alpha T}<\frac{\theta-\tilde{C}}{\hat{C}}.$$

It follows that $\theta-\tilde{C}-\hat{C}e^{\theta T}>0$, so inequality~\eqref{ineq} can hold if and only if, for all $t\in[0,T]$,
we have $r(t)=\bar r(t)$, $c(t)=\bar c(t)$, $q(t)=\bar q(t)$, $\phi_1(t)=\bar \phi_1(t)$, $\phi_2(t)=\bar \phi_2(t)$, and $\phi_3(t)=\bar \phi_3(t)$. But this is equivalent to
$R(t)=\bar R(t)$, $C(t)=\bar C(t)$, $P(t)=\bar P(t)$, $p_1(t)=\bar p_1(t)$, $p_2(t)=\bar p_2(t)$ and $p_3(t)=\bar p_3(t)$.

This establishes the uniqueness of the optimal control on the interval $[0,T]$.

We have two possibilities. If $T\ge t_f$, then we have uniqueness on the whole interval and we are done.
Otherwise, if $T < t_f$, considering the optimal control problem whose initial conditions
on time $T$ coincide with the values of the state variables on the end-time of the
interval $[0,T]$, we can obtain
uniqueness on $[T,2T]$ (note that, by the forward invariance of the set
$$\mathcal S=\{(C,P,R)\in (\R^+_0)^3: C+R+P \le C_0+R_0+P_0\},$$
and since the constants $\widetilde{C}$ and $\widehat{C}$ in \eqref{ineq}
depend only on the values of the several state and co-state variables on $\mathcal S$,
we still have the same $T$). Iterating the procedure, we conclude that we have uniqueness on the whole interval $[0,t_f]$, after a finite
number of steps. The proof is complete.
\end{proof}

\section{The optimal control problem with the $L^1$ objective functional.\label{appendixA}}
A quadratic objective favors lower rates: a  recruitment rate lower than the maximum, $u_{+}$, contributes with a value, much smaller, than $u_{+}$, to the cost (note that $u^2_{+}\ll u_{+}$). This feature is not related to the system, but is imposed by the choice of the functional and the maximum value of $u_{+}$. Hence, the linear objective, by incorporating the totality of controls, may be a more adequate choice.

Let us consider the linear objective:
\begin{equation}
\notag 
\mathcal{J}'(P,u_1,u_2) =\int_0^{t_f}\kappa_1\,P+ \kappa_2\, u_1+\kappa_3 \, u_2 ~dt
\end{equation}
where $0<\kappa_1,\kappa_2,\kappa_3 <\infty$ are weights that balance out the relative importance of the three terms. 

Using the adjoint variables $p=(p_1,p_2,p_3)$, the Hamiltonian of the linear objective functional and  system (\ref{eq:modelo}) is the following
\begin{equation}\notag 
\begin{split}
& \ham'(t,(C,R,P),(p_1,p_2,p_3),(u,v))\\
&  =  \kappa_1 P +\kappa_2 u_1+\kappa_3 u_2\\
& \quad + p_1(-\lbd_2 R+\lbd_1 C -\gamma(t) R+\alpha_1\, u_1 P +\alpha_2\, (\beta(t)+u_2) P R/N )\\
& \quad +p_2(-\lbd_1 C+\lbd_2 R -\gamma(t) C + (1-\alpha_2)(\beta(t)+u_2) PR/N + u_1(1-\alpha_1)P)\\
& \quad +p_3(- (\beta(t)+u_2) PR/N-u_1 P +\gamma(t) R +\gamma(t) C)
\end{split}
\end{equation}

We obtain the adjoint equations by
\[\dot p_1\parc{t} =  -\frac{\partial \ham'}{\partial R}, \quad \dot p_2\parc{t} =  -\frac{\partial \ham'}{\partial C}\; \text{ and }\;\dot p_3\parc{t} =  -\frac{\partial \ham'}{\partial P},\]
whose expressions are as in (\ref{eq:Pontryagin-CPR-Mayer-1}).

Once the terminal state, $(R(t_f),C(t_f),P(t_f))$, is free, the transversality conditions are again
\[ p_1(t_f)=p_2(t_f)=p_3(t_f)=0.\]

Since $\mathcal{H}'$ is linear in the control, this minimization problem can easily be solved~\cite{ledzewicz2011,preprint:Delfim2016}. Defining the so-called \emph{switching functions}, $\Phi_1$ and $\Phi_2$ as

\[\Phi_1(t)=\kappa_3+(\alpha_1 p_1(t) +(1-\alpha_1)p_2(t)-p_3(t))P(t)  \]
and
\[\Phi_2(t)=\kappa_3+(\alpha_2 p_1(t)+(1-\alpha_2)p_2(t)-p_3(t))P(t) R(t)/N. \]


Then the minimum condition for the optimal controls $u_1(t), u_2(t)$, is equivalent to the minimization problem $\Phi_i(t) u_i(t)={\min \atop {0\leqslant u_i\leqslant {u_i}_{\max}} } \Phi_i(t) u_i,i=1,2$.  This gives the following control functions

\begin{equation}u_i\parc{t}=\begin{cases}
  0 & \text{ if }\Phi_i(t)>0\\
  {u_i}_{\max} & \text{ if }\Phi_i(t)<0\\
\text{singular} & \text{ if }\Phi_i(t)=0 \text{ on an open subset of }[0,t_f]\end{cases},i=1,2.\label{eq:controlo_L1}\end{equation}

We do not discuss singular controls, since singular arcs never appeared in our computations. In view of the transversality conditions, the terminal values of the switching functions are: $\Phi_1(t_f)=\kappa_2$ and $\Phi_2(t_f)=\kappa_3$. According with the definition of control (\ref{eq:controlo_L1}), we may conclude that $u_1(t_f)=u_2(t_f)=0$, as with the quadratic functional.\\


\section*{Acknowledgment}
S. Rosa was partially supported by the Portuguese Science Foundation \linebreak
(FCT) through IT (project UID/EEA/50008/2013), P. Rebelo and C.M. Silva by FCT
through CMAUBI (project UID/MAT/00212/2013), H. Alves by FCT through
NECE  (PEst-OE/ EGE/UI0403/2014),  and P.G. Carvalho by FCT though CIDESD.

\bibliographystyle{elsart-num-sort}


\end{document}